\documentclass{amsart}
\usepackage{array}
\newcolumntype{P}[1]{>{\raggedright\let\newline\\\arraybackslash\hspace{0pt}}m{#1}}
\usepackage{graphicx,verbatim, amsmath, amssymb, amsthm, amsfonts, epsfig, amsxtra,ifthen,mathtools,epstopdf,caption,enumerate,hyperref,hhline,bbm,capt-of,longtable}	
\DeclareFontFamily{U}{mathx}{\hyphenchar\font45}
\DeclareFontShape{U}{mathx}{m}{n}{<-> mathx10}{}
\DeclareSymbolFont{mathx}{U}{mathx}{m}{n}
\DeclareMathAccent{\widebar}{0}{mathx}{"73}
\DeclareGraphicsExtensions{.ps}
\newtheorem{proposition}{Proposition}[section]
\newtheorem{theorem}[proposition]{Theorem}

\newtheorem{lemma}[proposition]{Lemma}
\newtheorem{prop}[proposition]{Proposition}
\newtheorem{cor}[proposition]{Corollary}
\newtheorem{thm}[proposition]{Theorem}
\newtheorem{conj}[proposition]{Conjecture}

\theoremstyle{definition}

\newtheorem{definition}[proposition]{Definition}

\theoremstyle{remark}
\newtheorem{remark}[proposition]{Remark}

\numberwithin{equation}{section}
\usepackage[usenames]{color}

\usepackage{color}

\newcommand{\margincolor}{red}      
\definecolor{darkgreen}{rgb}{0,0.7,0}

\newcounter{margincounter}
\newcommand{\marginnum}{
\ifnum\value{margincounter}<10
\textcolor{\margincolor}{\begin{picture}(0,0)\put(2.2,2.4){\circle{9}}\end{picture}\footnotesize\arabic{margincounter}}
\else\ifnum\value{margincounter}<100
\textcolor{\margincolor}{\begin{picture}(0,0)\put(4.256,2.5){\circle{11}}\end{picture}\footnotesize\arabic{margincounter}}
\else
\textcolor{\margincolor}{\begin{picture}(0,0)\put(6.8,2.5){\circle{14}}\end{picture}\footnotesize\arabic{margincounter}}
\fi\fi
}

\newcommand{\newword}[1]{\textbf{\emph{#1}}}

\newcommand{\integers}{\mathbb Z}
\newcommand{\rationals}{\mathbb Q}

\newcommand{\reals}{\mathbb R}

\newcommand{\ep}{\varepsilon}
\newcommand{\thet}{\vartheta}

\newcommand{\sgn}{\operatorname{sgn}}

\newcommand{\rank}{\operatorname{rank}}

\newcommand{\Ram}{{\operatorname{Ram}}}
\newcommand{\uf}{{\operatorname{uf}}}
\newcommand{\fr}{{\operatorname{fr}}}

\newcommand{\set}[1]{{\left\lbrace #1 \right\rbrace}}

\newcommand{\br}[1]{{\langle #1 \rangle}}

\newcommand{\E}{{\mathcal E}}

\newcommand{\F}{{\mathcal F}}
\newcommand{\D}{{\mathfrak D}}

\newcommand{\p}{{\mathfrak p}}

\renewcommand{\th}{^\text{th}}

\newcommand{\relint}{\mathrm{relint}}

\DeclareMathOperator{\Span}{Span}

\newcommand{\g}{\mathbf{g}}
\newcommand{\s}{\mathbf{s}}
\newcommand{\m}{\mathbf{m}}

\renewcommand{\b}{\mathbf{b}}
\renewcommand{\k}{\mathbbm{k}}
\newcommand{\kk}{\mathbf{k}}
\newcommand{\ks}{\mathbf{k}}

\newcommand{\e}{\mathbf{e}}

\newcommand{\tB}{\tilde{B}}

\newcommand{\C}{\mathcal{C}}
\newcommand{\B}{\mathcal{B}}

\newcommand{\R}{\mathcal{R}}

\renewcommand{\H}{\mathcal{H}}

\newcommand{\Supp}{\operatorname{Supp}}

\newcommand{\Clear}{\operatorname{Clear}}

\newcommand{\Hom}{\operatorname{Hom}}
\newcommand{\Scat}{\operatorname{Scat}}
\newcommand{\Fan}{\operatorname{Fan}}
\newcommand{\ScatFan}{\operatorname{ScatFan}}
\newcommand{\ChamberFan}{\operatorname{ChamberFan}}

\newcommand{\can}{\operatorname{can}}

\renewcommand{\d}{{\mathfrak d}}
\newcommand{\seg}[1]{\overline{#1}}

\allowdisplaybreaks

\title{Scattering fans}
\author{Nathan Reading}
\address{Department of Mathematics, North Carolina State University}
\email{reading@math.ncsu.edu}
\thanks{Partially supported by the National Science Foundation under Grant Number DMS-1500949.}

\begin{document}

\begin{abstract}
Scattering diagrams arose in the context of mirror symmetry, Donaldson-Thomas theory, and integrable systems.
We show that a consistent scattering diagram with minimal support cuts the ambient space into a complete fan.
A special class of scattering diagrams, the cluster scattering diagrams, are closely related to cluster algebras.
We show that the cluster scattering fan associated to an exchange matrix $B$ refines the mutation fan for~$B$ (a complete fan that encodes the geometry of mutations of~$B$).
We conjecture that, when $B$ is $n\times n$ for $n>2$, these two fans coincide if and only if $B$ is of finite mutation type.
\end{abstract}
\maketitle


\setcounter{tocdepth}{1}
\tableofcontents

\section{Introduction}\label{intro sec}  
Scattering diagrams (also known as wall-crossing structures) are combinatorial/discrete geometric objects that arise in mirror symmetry, Donaldson-Thomas theory, and integrable systems.
History and background is available in \cite{GHKK,KS}.
In this paper, we review the construction of scattering diagrams from \cite{GHKK}, with essentially the same conventions, except that we ignore some unnecessary dimensions, as explained in Remark~\ref{useless dimensions}.
Our first main result (Theorem~\ref{scat fan}) is that a consistent scattering diagram cuts the ambient vector space into a complete fan.
The definition of a fan here is the usual notion, except that it allows for infinitely many cones.
This broader definition allows stranger behaviors than occur for finite fans.
See Section~\ref{scatfan sec} and particularly Remark~\ref{crazy fan}.

Cluster scattering diagram \cite{GHKK} are certain special scattering diagrams that are deeply connected to cluster algebras.
Our second main result (Theorem~\ref{scat ref mut}) is that the cluster scattering fan for an exchange matrix $B$ refines the mutation fan for $B$.
The mutation fan for $B$ is a complete fan that encodes the piecewise linear geometry of matrix mutations of $B$, in the sense of \cite[Definition~4.2]{ca1}.
Theorem~\ref{scat ref mut} follows from a sort of universal property (Proposition~\ref{direct construct FB}) of the mutation fan:
Roughly, any family of fans that ``mutates in the right way'' refines the mutation fans.
We recast a result of \cite{GHKK} to show that cluster scattering fans mutate in the right way, and conclude that cluster scattering fans refine mutation fans.
We conjecture that these two fans coincide if and only if either $B$ is $2\times2$ and of finite or affine type or $B$ is $n\times n$ for $n>2$ and of finite mutation type (Conjecture~\ref{scat eq mut}).

We conclude by observing (Theorem~\ref{clus mon pop}) that the cluster monomials---and in particular the cluster variables---can be expressed simply in terms of path-ordered products in the cluster scattering fan.
This observation is a crucial ingredient in~\cite{scatcomb}.
Although it is an easy consequence of results of \cite[Section~3]{GHKK}, it has apparently not been stated before in full generality, but a version for $2\times2$ exchange matrices is mentioned in \cite[Remark~3.10]{CGMMRSW}.

\section{Scattering diagrams}\label{scat sec}
In this section, we quote the definition of a scattering diagram and quote and prove some basic facts about scattering diagrams.
We work with the general construction of scattering diagrams from \cite{GHKK}, as opposed to the more special construction of \emph{cluster} scattering diagrams, which make their appearance in Section~\ref{clus scat sec}.

In defining scattering diagrams, we follow \cite{GHK,GHKK}, with a few modifications.
A scattering diagram depends on the initial data and related definitions presented in Table~\ref{init data}, below.
In \cite{GHK,GHKK}, the data are divided into two types: fixed data and seed data.
For us, that distinction is less important.
(Indeed, later in Section~\ref{mut sec} we take a rather different point of view on what is fixed and what is not.)
We include in the table some definitions that will not be used here but that might be useful for comparing with \cite{GHKK}.

\renewcommand*{\arraystretch}{1.29}
\setlength{\doublerulesep}{0pt}
\begin{longtable}{P{88.5pt}|P{248pt}}
\caption{Initial data and preliminary definitions for scattering diagrams}\label{init data}\\
Notation&Description/requirements\\\hline\hline\hline
\endfirsthead
\caption{(continued)}\\
Notation&Description/requirements\\\hline\hline\hline
\endhead
$N$& finite-dimensional lattice\\\hline
$M=\Hom(N,\integers)$&dual lattice to $N$\\\hline
$N_\uf\subseteq N$&saturated sublattice of $N$\\\hline
$V$&$N_\uf\otimes\reals$, real ambient vector space of $N_\uf$ \\\hline
$I$&an index set with $|I|=\rank N$\\\hline
$I_\uf\subseteq I$&$|I_\uf|=\rank N_\uf$\\\hline
$I_\fr\subseteq I$&$I_\fr=I\setminus I_\uf$\\\hline
$\set{\,\cdot\,,\,\cdot\,}:N\times N\to\rationals$&skew-symmetric bilinear form\\\hline
$N^\circ\subseteq N$&{sublattice of finite index with $\set{N_\uf,N^\circ}\subseteq\integers$ and $\set{N,N_\uf\cap N^\circ}\subseteq\integers$}\\\hline
$M^\circ=\Hom(N^\circ,\integers)$&dual lattice to $N^\circ$, a finite-index superlattice of $M$\\\hline
$d_i$&positive integers indexed by $I$ with $\gcd_{i\in I}(d_i)=1$\\\hline
$\s=(e_i:i\in I)$&$(e_i:i\in I)$ is a basis for $N$, $(e_i:i\in I_\uf)$ is a basis for $N_\uf$, and $(d_ie_i:i\in I)$ is a basis for $N^\circ$\\\hline
$N^+=N^+_\s$&$\set{\sum_{i\in I_\uf}a_ie_i:a_i\in\integers,\,a_i\ge0,\,\sum_{i\in I_\uf}a_i>0}\subset N_\uf$\\\hline
$[\,\cdot\,,\,\cdot\,]_\s:N\times N\to\rationals$&$[e_i,e_j]_\s=\set{e_i,e_j}d_j$\\\hline
$\epsilon_{ij}$&$[e_i,e_j]_\s$, an integer except possibly when $i,j\in I_\fr$\\\hline
$(e^*_i:i\in I)$&basis for $M$ dual to $(e_i:i\in I)$\\\hline
$(f_i:i\in I)$&basis for $M^\circ$ given by $f_i=d_i^{-1}e_i^*$\\\hline
$V^*$&real span of $\set{e_i^*:i\in I_\uf}$, dual space to $V$ \\\hline
$\br{\,\cdot\,,\,\cdot\,}:V^*\times V\to\reals$\newline$\br{\,\cdot\,,\,\cdot\,}:M^\circ\times N\to\rationals$&natural pairing\\\hline
$p^*:N_\uf\to M^\circ$&$\br{p^*(e_i),d_je_j}=\epsilon_{ij}$ for all $i\in I_\uf$ and $j\in I$\newline 
$(p^*(e_i):i\in I_\uf)$ required to be linearly independent\\\hline
$(v_i:i\in I)$&$v_i=p^*(e_i)$ for $i\in I_\uf$ and
$v_i\in M^\circ$ for $i\in I_\fr$ chosen to make $(v_i:i\in I)$ linearly independent\\\hline
$z_i$&indeterminates indexed by $I$\\\hline
$z^m$&$\prod_{i\in I}z_i^{c_i}$ for $m=\sum_{i\in I}c_if_i\in M^\circ$\\\hline
$(\zeta_i:i\in I)$&$\zeta_i=z^{v_i}$\\\hline
$\zeta^n$&$\prod_{i\in I}\zeta_i^{c_i}$ for $n=\sum_{i\in I}c_ie_i\in N$\\\hline
$\k$&a field of characteristic zero\\\hline
$\k[[\zeta]]$&$\k[[\zeta_i:i\in I]]$, ring of power series in the $\zeta_i$\\\hline
$\m$&ideal in $\k[[\zeta]]$ consisting of series with constant term zero\\\hline
\end{longtable}

The subscript ``$\fr$'' in the table stands for \newword{frozen} and ``$\uf$'' stands for \newword{unfrozen}.
The values $\epsilon_{ij}$ for $i,j\in I_\fr$ are unimportant for our purposes.
One could, for example, take all of these values to be zero.
The requirement that the $p^*(e_i)$, for $i\in I_\uf$, are linearly independent is a condition on $\set{\,\cdot\,,\,\cdot\,}$.
The elements $v_i$ are vectors whose nonnegative span is called~$\sigma$ in \cite{GHKK}.
The exact choice of the $v_i$ for $i\in I_\fr$ does not affect the construction of the scattering diagram or the important notion of consistency, defined below.
Here we are taking $\sigma$ to be a specific simplicial cone whereas more freedom is given in \cite{GHKK}, but this extra freedom is not meaningful for our purposes.  
The monomials we call $\zeta_i$ are not named in \cite{GHKK}.

The algebraic setting for the scattering diagram is the multivariate formal power series ring $\k[[\zeta]]$ or the quotient $\k[[\zeta]]/\m^{k+1}$ for $k\ge0$.
We will introduce scattering diagrams in both settings simultaneously, for two reasons:
first, because the definitions are the same, except for working modulo $\m^{k+1}$; and second, because the option to work modulo $\m^{k+1}$ is essential to the construction in the setting of $\k[[\zeta]]$.
(Both of these settings fit into a broader level of generality discussed in \cite[Remark~1.7]{GHKK}).
Working modulo $\m^{k+1}$ amounts to setting to zero all monomials in $\set{\zeta_i:\in I_\uf}$ of total degree greater than $k$.

A \newword{rational subspace} in $V^*$ is the intersection of a finite collection of \newword{rational hyperplanes} $\set{v\in V^*:\br{v,n}=0}$ for various $n\in N_\uf$.
A \newword{rational cone} in $V^*$ is the intersection of a finite collection of halfspaces of the form $\set{v\in V^*:\br{v,n}\le0}$ for various $n\in N_\uf$.
A \newword{wall} is a pair $(\d,f_\d)$, where $\d$ is a codimension-$1$ rational cone in $V^*$, contained in $n_0^\perp$ for some \emph{primitive} $n_0\in N^+$, and $f_\d$ is $1+\sum_{\ell\ge1}c_\ell\bigl(z^{p^*(n_0)}\bigr)^\ell$ with coefficients $c_\ell$ in $\k$, considered modulo $\m^{k+1}$ when appropriate.
(Requiring $n_0$ to be primitive in $N^+$ means that there is no $c$ with $0<c<1$ such that $cn_0\in N$.)
By definition, $f_\d$ is a univariate formal power series in $z^{p^*(n_0)}$, but also since $n_0\in N^+$, the term $z^{p^*(n_0)}=\zeta^{n_0}$ is a monomial in the $\zeta_i$, and thus $f_\d$ is in $\k[[\zeta]]$.
It will sometimes be convenient to write $f_\d=f_\d(\zeta^{n_0})$ as a way to name the normal vector $n_0$ explicitly.
 
We call a wall $(\d,f_\d)$ \newword{incoming} if $p^*(n_0)\in\d\oplus\Span_\reals\set{f_i:i\in I_\fr}$.  
Otherwise $(\d,f_\d)$ is \newword{outgoing}.
Say that two walls are \newword{parallel} if they are contained in the same hyperplane.

A \newword{scattering diagram} is a collection $\D$ of walls, satisfying a finiteness condition that we now explain.
Write $\D_k$ for the set of walls $(\d,f_\d)\in\D$ such that $f_\d\not\equiv 1$ modulo $\m^{k+1}$.
The condition is that $\D_k$ is a finite set for all $k\ge 1$.
In particular, $\D=\bigcup_{k\ge1}\D_k$ is a countable collection of walls.
The \newword{support} $\Supp(\D)$ of $\D$ is $\bigcup_{(\d,f_\d)\in\D}\d$, the union of all the walls of $\D$.

\begin{remark}\label{useless dimensions teaser}
We construct scattering diagrams in $V^*$, while \cite{GHKK} constructs scattering diagrams in a larger vector space.
We make simple modifications in the remaining definitions to account for the loss of these dimensions.
(Indeed, we have already modified the definition of incoming walls.)
The omission of these dimensions is justified later in Remarks~\ref{useless dimensions} and~\ref{coeffs don't matter remark}.
\end{remark}

Suppose $\D$ is a scattering diagram.
A piecewise differentiable path $\gamma:[0,1]\to V^*$ is \newword{generic} for $\D$ if it satisfies the following conditions:
\begin{itemize}
\item $\gamma$ does not pass through the intersection of any two non-parallel walls of $\D$.
\item $\gamma$ does not pass through the relative boundary of any wall (its boundary in the hyperplane it spans).
\item Neither $\gamma(0)$ nor $\gamma(1)$ is contained in a wall of $\D$.
\item When $\gamma$ intersects a wall, it crosses the wall transversely (although $\gamma$ may not be differentiable where it intersects the wall).
\end{itemize}

\begin{prop}\label{generic exists}
For any $p,q\in V^*\setminus\Supp(\D)$, there exists a generic path from $p$ to $q$.
\end{prop}
\begin{proof}
Since each wall of $\D$ is in $n_0^\perp$ for some $n_0\in N^+$, there is no wall of $\D$ contained in a hyperplane that intersects the interior of the positive cone in $V^*$ (the full-dimensional cone spanned by the $e^*_i$ for $i\in I_\uf$).

Fix $p,q\in V^*\setminus\Supp(\D)$.
We observe that there exists a point $r$ in the interior of the positive cone such that the straight line segment $\seg{pr}$ is a generic path.
Indeed, if $\seg{pr}$ fails to be generic for some $r$, then either (1) it passes through the intersection of nonparallel walls or the relative boundary of a wall, or (2) it is contained in the hyperplane containing some wall.
But (2) can't happen when $r$ is in the interior of the positive cone.
We can rule out (1) by choosing $r$ to avoid a countable union of proper affine subspaces of $V^*$.
(Specifically, $r$ must not be in the affine span of $p$ and an intersection of nonparallel walls or in the affine span of $p$ and a proper face of a wall.)
Similarly, we find a generic line segment $\seg{r'q}$ for some $r'$ in the positive cone.
The concatenation of the segments $\seg{pr}$, $\seg{rr'}$, and $\seg{r'q}$ is the desired path.
\end{proof}

Suppose $\gamma$ is a generic path for $\D$.
If $\gamma$ crosses a wall $\d\subseteq n_0^\perp$, specifically with $\gamma(t)$ contained in the wall, the \newword{wall-crossing automorphism} $\p_{\gamma,\d}$ of $\k[[\zeta]]$ is
\begin{equation}\label{p def}
\p_{\gamma,\d}(\zeta_i)=\zeta_if_\d^\br{v_i,\pm n'_0},
\end{equation}
where $n_0'$ is the normal vector to $\d$ that is contained in $N^+$ and is primitive in~$N^\circ$.
We take $+n_0'$ in the formula when $\br{\gamma'(t),n_0}<0$ or $-n_0'$ when ${\br{\gamma'(t),n_0}>0}$.
If~$\gamma$ is not differentiable at $t$, we abuse notation and still write ``$\br{\gamma'(t),n_0}<0$'' and ``${\br{\gamma'(t),n_0}>0}$'' to describe the two directions $\gamma$ can cross the wall transversely.

A few comments on wall-crossing automorphisms are in order.
First, and most importantly, we emphasize a subtlety in the definition of $\p_{\gamma,\d}$:
Even though the function $f_\d$ is a formal power series in $z^{n_0}$, where $n_0\in N^+$ is primitive in $N$, the exponent on $f_\d$ in the formula depends on $n_0'$, a positive integer multiple of $n_0$ which is primitive in $N^\circ$.
This is precisely where, later on when we consider the cluster scattering diagram associated to an exchange matrix $B$, the theory allows $B$ to be \emph{skew-symmetrizable}, rather than requiring that $B$ be \emph{skew-symmetric}.
Second, $\p_{\gamma,\d}$ is indeed an automorphism, with inverse given by $\zeta_i\mapsto \zeta_if_\d^{-\br{v_i,\pm n'_0}}$, which is the wall-crossing automorphism associated to crossing the wall in the opposite direction.
We see that this inverse map sends $\zeta_if_\d^\br{v_i,\pm n'_0}$ to $\zeta_i$ as required because $f_\d$ is a sum of powers of $\zeta^{n_0}=z^{p^*(n_0)}$ and $\br{p^*(n_0),n_0'}=\set{n_0,n_0'}=0$ by the skew-symmetry of $\set{\,\cdot\,,\,\cdot\,}$.

For each $k\ge 1$, we write $\p_{\gamma,\D_k}$ for the automorphism $\p_{\gamma,\d_\ell}\circ\cdots\circ\p_{\gamma,\d_1}$ of $\k[[\zeta]]$ such that $\set{\d_1,\ldots,\d_\ell}$ is the sequence of walls of $\D_k$ crossed by $\gamma$, with wall $\d_i$ crossed at time $t_i$ and with $t_1\le t_2\le\cdots\le t_\ell$.
This is well-defined because if $t_i=t_{i+1}$, the functions $\p_{\gamma,\d_i}$ and $\p_{\gamma,\d_{i+1}}$ commute.
Indeed, if $t_1=t_2=\cdots=t_\ell$, then $\p_{\gamma,\d_\ell}\circ\cdots\circ\p_{\gamma,\d_1}$ sends $\zeta_i$ to $\zeta_i\bigl(\prod_{i=1}^\ell f_{\d_i}\bigr)^\br{v_i,\pm n'_0}$.
(This is verified by an argument similar to the argument above on the inverse of $\p_{\gamma,\d}$.)
We define the \newword{path-ordered product} $\p_{\gamma,\D}:\k[[\zeta]]\to\k[[\zeta]]$ to be the limit of $\p_{\gamma,\D_k}$ as $k\to\infty$.
A scattering diagram $\D$ is \newword{consistent} if each path-ordered product $\p_{\gamma,\D}$ depends only on the endpoints $\gamma(0)$ and $\gamma(1)$.

We make some useful observations about consistency.
The first observation is that, when $\D$ is consistent, in light of Proposition~\ref{generic exists}, one can define $\p_{p,q,\D}$ for any $p,q\in V^*\setminus\Supp(\D)$ to be $\p_{\gamma,\D}$ for any generic path from $p$ to $q$.
If $p,q,r\in V^*\setminus\Supp(\D)$, then $\p_{q,r,\D}\circ\p_{p,q,\D}=\p_{p,r,\D}$.

The second observation is the following proposition, which makes it slightly easier to compute path-ordered products in consistent scattering diagrams.
\begin{prop}\label{easier limit}
Suppose $\D$ is a consistent scattering diagram and suppose $p$ and $q$ are in $V^*\setminus\Supp(\D)$.
Suppose for each $k\ge 1$ that $\gamma_k$ is a path from $p$ to $q$ that is generic in $\D_k$.
Then $\p_{p,q,\D}$ agrees, modulo $\m^{k+1}$, with $\p_{\gamma_k,\D_k}$ for each $k\ge1$.
Thus $\p_{p,q,\D}$ is the limit, as $k\to\infty$, of $\p_{\gamma_k,\D_k}$.
\end{prop}
When we say that $\p_{p,q,\D}$ and $\p_{\gamma_k,\D_k}$ agree modulo $\m^{k+1}$, we mean that $\p_{p,q,\D}$ descends to a function $\k[[\zeta]]/{\m^{k+1}}\to\k[[\zeta]]/{\m^{k+1}}$ that coincides with $\p_{\gamma_k,\D_k}$.
The statement is true by definition if each $\gamma_k$ is generic for $\D$.
The point of the proposition is to allow us to merely check genericity of $\gamma_k$ for $\D_k$.
\begin{proof}
Each path $\gamma_k$ crosses some sequence of walls of $\D_k$ (possibly crossing some parallel walls simultaneously).
Choose points $p_0,\ldots,p_\ell$ on $\gamma_k$ and in $V^*\setminus\Supp(\D)$ between all distinct intersection points of $\gamma_k$ with walls of $\D_k$, taking $p_0=p$ and $p_\ell=q$.
Then $\p_{p,q,\D}=\p_{p_{\ell-1},p_\ell,\D}\circ\cdots\circ\p_{p_0,p_1,\D}$.
This agrees modulo $\m^{k+1}$ with $\p_{p,q,\D_k}=\p_{p_{\ell-1},p_\ell,\D_k}\circ\cdots\circ\p_{p_0,p_1,\D_k}$.
But by cutting $\gamma_k$ into shorter paths from $p_{i-1}$ to $p_i$ for each $i=1,\ldots,\ell$, we see that $\p_{p,q,\D_k}=\p_{\gamma_k,\D_k}$, and we conclude that $\p_{p,q,\D}$ and $\p_{\gamma_k,\D_k}$ agree modulo $\m^{k+1}$ for all $k\ge1$.
\end{proof}

The third observation is the following proposition, which uses the previous proposition to connect the consistency of $\D$ to the consistency of each $\D_k$.
\begin{prop}\label{Dk consist}
A scattering diagram $\D$ is consistent if and only if each $\D_k$ is consistent (modulo $\m^{k+1}$).
\end{prop}
\begin{proof}
Suppose $\D$ is consistent and suppose $\gamma_k$ is a path from $p$ to $q$, generic in $\D_k$.
If $p,q\not\in\Supp(\D)$, then Proposition~\ref{easier limit} says that $\p_{\gamma_k,\D_k}=\p_{p,q,\D}$ modulo $\m^{k+1}$, so in particular  $\p_{\gamma_k,\D_k}$ depends only on $p$ and $q$.
If $p$ and/or $q$ is in $\Supp(\D)$, then since $\D_k$ is finite, by taking $p'$ and/or $q'$ in small enough balls about $p$ and/or $q$, we can extend $\gamma_k$ by pre-appending a path from $p'$ to $p$ and/or post-appending a path from $q$ to $q'$ without changing $\p_{\gamma_k,\D_k}$.
Then we apply Proposition~\ref{easier limit} as above.

Conversely, if each $\D_k$ is consistent, then in the definition of $\p_{\gamma,\D}$, each $\p_{\gamma,\D_k}$ depends only on the endpoints of $\gamma$, so $\D$ is consistent.
\end{proof}

For the fourth observation, we re-use the definition above verbatim to reinterpret each path-ordered product $\p_{\gamma,\D}$ as a map sending each Laurent monomial $z^m$ (for $m\in M^\circ)$ to $z^m$ times an element of $\k[[\zeta]]$.
In this interpretation, crossing a wall $(\d,f_\d(\zeta^{n_0}))$ fixes $z_i$ for $i\in I_\fr$ and for $m\in M^\circ\cap V^*$, sends $z^m$ to $z^mf_\d^\br{m,\pm n_0'}$ (with $n_0'$ and the choice of sign for $\pm n_0'$ as in the definition).
The following proposition is immediate because of the linearity in both interpretations.

\begin{prop}\label{just z's}    
A path-ordered product $\p_{\gamma,\D}$ is determined entirely by its values $\p_{\gamma,\D}(z_i)$ for $i\in I_\uf$, or equivalently by its values $\p_{\gamma,\D}(z^m)$ for $m\in M^\circ\cap V^*$.
\end{prop}

The fifth and final observation is that, in some sense, consistence of a scattering diagram only depends on ${[\epsilon_{ij}]_{i,j\in I_\uf}}$.
See also Remark~\ref{coeffs don't matter remark}.
The matrix ${[\epsilon_{ij}]_{i,j\in I_\uf}}$ is called an \newword{exchange matrix} and will have a more prominent role beginning in Section~\ref{clus scat sec}.

\begin{prop}\label{coeffs don't matter}
Take two choices of initial data as in Table~\ref{init data}, agreeing to the extent that $I_\uf$, the $e_i$ for $i\in I_\uf$, the $d_i$ for $i\in I_\uf$, and the restriction of $\br{\,\cdot\,,\,\cdot\,}$ to $N_\uf\times N_\uf$ are the same for both choices.
Distinguish the two choices by placing ``primes'' on the notation for one choice, writing for example $I$ and $I'$ for the two indexing sets.
If $\D$ is a consistent scattering diagram for the unprimed data, then a consistent scattering diagram for the primed data is obtained by replacing each $(\d,f_\d(\zeta^{n_0}))\in\D$ by $(\d,f_\d((\zeta')^{n_0}))$.   
\end{prop}
We emphasize that $I$ and $I'$ are not assumed to coincide, beyond the fact that $I_\uf=I'_\uf$.
Since $(e_i:i\in I_\uf)$ and $(e'_i:i\in I_\uf)$ coincide, we have $N_\uf=N'_\uf$, ${N^+=(N')^+}$ and $V^*=(V')^*$, but $N$ and $N'$ need not even have the same dimension.
\begin{proof}
The agreement between unprimed and primed data is enough to imply that the quantities $\br{v_i,n}$ and $\br{v'_i,n}$ agree for each $n\in N^+=(N')^+$.
Thus for each $m\in M^\circ\cap V^*$ and generic path $\gamma$, the wall-crossing maps applied in computing $\p_{\gamma,\D}(z^m)$ commute with replacing each $\zeta_i$ by $\zeta'_i$ for $i\in I_\uf=I'_\uf$.
By Proposition~\ref{just z's}, we are done.
\end{proof}

Two scattering diagrams $\D$ and $\D'$ are \newword{equivalent} if and only if for each path $\gamma$ that is generic for both $\D$ and $\D'$, the path-ordered products $\p_{\gamma,\D}$ and $\p_{\gamma,\D'}$ coincide.
A point $p\in V^*$ is \newword{general} if there is at most one $n_0\in N^+$ such that $p\in n_0^\perp$.
A scattering diagram $\D$ defines a function $p\mapsto f_p(\D)=\prod_{\d\ni p}f_\d\in\k[[\zeta]]$ on general points $p$.
The following is \cite[Lemma~1.9]{GHKK}.  
\begin{lemma}\label{lem1.9}
Two scattering diagrams $\D$ and $\D'$ are equivalent if and only if $f_p(\D)=f_p(\D')$ for all general points $p$.
\end{lemma}

%

It is possible for a scattering diagram $\D$ to have walls, or parts of walls, that are irrelevant.
For example, a wall $(\d,f_\d)$ may have $f_\d=1$, so that the wall can be ignored in all computations of path-ordered products.
More subtly, two overlapping, parallel walls $(\d_1,f_{\d_1})$ and $(\d_2,f_{\d_2})$ might have $f_{\d_1}f_{\d_2}=1$, so that any path through their intersection would not ``see'' them.
Even more subtly, a collection $\set{(\d_i,f_{\d_i}):i\ge1}$ of parallel walls with $\bigcap_{i\ge1}\d\neq\emptyset$ can also be invisible to paths through their intersection even if no finite product of the $f_{\d_i}$ is~$1$.
As a concrete example, if all of these walls are in $n_0^\perp$, inductively choose $f_{\d_i}$ so that $f_{\d_1}\cdots f_{\d_i}=1+\zeta^{in_0}$ and thus $\prod_{i\ge1}f_{\d_i}=1$.
To avoid having irrelevant walls, we consider scattering diagrams with \newword{minimal support}, meaning that no scattering diagram $\D'$ equivalent to $\D$ has $\Supp(\D')\subsetneq\Supp(\D)$.

\begin{prop}\label{min sup}
Every consistent scattering diagram $\D$ is equivalent to a scattering diagram $\D'$ with minimal support and with the property that each $\D'_k$ has minimal support (as a scattering diagram modulo $\m^{k+1}$).
All scattering diagrams equivalent to $\D$ and having minimal support have the same support.
\end{prop}
\begin{proof}  
We will construct a scattering diagram $\D'$ equivalent to $\D$ such that, for each $k\ge1$, $\Supp(\D'_k)$ is the closure of the set of general points $p$ such that $f_p(\D)\neq1$ modulo $\m^{k+1}$.
We begin by taking $\D'_0$ to be the empty scattering diagram.
Now for each $k$, suppose that $\D'_{k-1}$ has been defined to be equivalent to $\D_{k-1}$ and has $\Supp(\D'_{k-1})$ as desired.
Then the function $p\mapsto f_p(\D_k)/f_p(\D'_{k-1})$ modulo $\m^{k+1}$ on general points $p$ is determined by finitely many walls (the walls of $\D_k$ and the walls of $\D'_{k-1}$.
Thus we can add finitely many walls to $\D'_{k-1}$ to define $\D'_k$ with support as desired.
The union over $k\ge1$ of the $\D'_k$ is the desired $\D'$.
This is equivalent to $\D$ because it is equivalent modulo $\m^{k+1}$ for every $k$.
If $\D''$ is equivalent to $\D$ and has minimal support, then Lemma~\ref{lem1.9} says that $f_p(\D'')=f_p(\D)$ for all general points.
In particular, $\Supp(\D''_k)$ contains all points where $f_p(\D)\neq1$ modulo $\m^{k+1}$, and thus $\Supp(\D''_k)$ contains $\Supp(\D'_k)$.
Therefore $\Supp(\D'')$ contains $\Supp(\D')$.
\end{proof}

We will need the following technical lemmas.
Given a scattering diagram $\D$ and $n_0\in N^+$, the \newword{rampart} of~$\D$ associated to $n_0$ is the union of all walls of $\D$ contained in $n_0^\perp$.
Note that we do not assume consistency of $\D$ in the following lemma.

\begin{lemma}\label{tech lem}
Suppose that $\D$ has minimal support, that each $\D_k$ has minimal support (as a scattering diagram modulo $\m^{k+1}$), and that $\gamma$ is a generic path in $\D$ that crosses no rampart of $\D$ more than once.
Then $\p_{\gamma,\D}$ is the identity if and only if $\gamma$ crosses no wall in $\D$.
\end{lemma}
\begin{proof}
The ``if'' direction holds by definition.
To prove the ``only if'' direction, suppose $\gamma$ crosses a wall of $\D$.
Let $k$ be minimal such that $\gamma$ crosses a wall of $\D_k$.
Then $\gamma$ crosses (necessarily all at the same time) some collection of parallel walls such that the product $f$ of the functions on the wall is $1+c\zeta^{n_0}$ plus higher-order terms in $\zeta^{n_0}$ for some $c\neq 0$ and $n_0\in N^+$ such that the total degree of $\zeta^{n_0}$ is $k$.

For each $i\in I$, we have $\p_{\gamma,\d}(\zeta_i)=\zeta_if_\d^{\br{v_i,\pm n_0}}$, with the sign in $\pm n_0$ chosen as explained in the definition.
Since the $v_i$ are linearly independent, we can choose $i\in I$ such that $\br{v_i,\pm n_0}\neq0$, so that $\p_{\gamma,\d}(\zeta_i)=\zeta_i(1+a\zeta^{n_0}+\cdots)$, where $a$ is $c\br{v_i,\pm n_0}\neq0$ and the ``$\cdots$'' represents higher-order terms in $\zeta^{n_0}$.

We partially order $N^+\cup\set{0}$ componentwise according to the basis $(e_i:i\in I_\uf)$ and consider the closed interval $[0,n_0]$ in this order.
We observe that $\zeta_i^{-1}\p_{\gamma,\D_k}(\zeta_i)\in \k[[\zeta]]$ and consider for which vectors $n\in[0,n_0]$ the monomial $\zeta^n$ appears with nonzero coefficient in each step in evaluating $\zeta_i^{-1}\p_{\gamma,\D_k}(\zeta_i)$.
Before any walls are crossed, only the monomial $\zeta^0=1$ appears, and then the first change that occurs is to insert the monomial $a\zeta^{n_0}$.
This monomial could be immediately removed if $\gamma$ again crossed a wall in $n_0^\perp$, but since $\gamma$ does not cross any rampart of $\D$ twice, that does not happen.
The only change that can occur next is to pick up a term $\zeta^{n_1}$ for some $n_1$ strictly between $0$ and $n_0$.
Once we have this new term, it is possible to lose the term $\zeta^{n_0}$.
However, continuing onward, we consider the smaller interval $[0,n_1]$, first noticing that the next change cannot be to lose the term $\zeta^{n_1}$.
Proceeding in this manner, we conclude that $\zeta_i^{-1}\p_{\gamma,\D_k}(\zeta_i)\neq 1$, so that $\p_{\gamma,\D_k}$ is not the identity map on $\k[[\zeta]]/\m^{k+1}$.
But $\p_{\gamma,\D_k}$ agrees with $\p_{\gamma,\D}$ modulo $\m^{k+1}$, so $\p_{\gamma,\D}$ is not the identity.
\end{proof}

We make a slightly more detailed statement of Lemma~\ref{tech lem} in a special case.
\begin{lemma}\label{tech lem detailed}
Suppose $\D$ has minimal support, that each $\D_k$ has minimal support (as a scattering diagram modulo $\m^{k+1}$), and that $\gamma$ is a generic path in $\D$ that is a line segment.
If $\gamma$ crosses a wall contained in $n_0^\perp$ for some $n_0\in N^+$ and if $m\in M^\circ\setminus n_0^\perp$, then $\p_{\gamma,\D}(z^m)\neq z^m$.
\end{lemma}
\begin{proof}
In the proof of Lemma~\ref{tech lem}, consider $z^{-m}\p_{\gamma,\D_k}(z^m)$ instead of $\zeta_i^{-1}\p_{\gamma,\D_k}(\zeta_i)\in \k[[\zeta]]$ and conclude that it is not equal to $1$.
\end{proof}

\begin{lemma}\label{bare point}
Suppose that $\D$ is a consistent scattering diagram with minimal support, that $p$ is a point in a rampart $R$ of $\D$, and that $v\in V^*$ has the property that there exists $E>0$ such that $p+\ep v\not\in R$ whenever $E>\ep>0$.
Then $p$ is contained in some wall of $\D$ whose hyperplane does not contain $p+v$.
\end{lemma}
\begin{proof}
By Proposition~\ref{min sup}, up to equivalence and without changing the support of~$\D$, we assume that each $\D_k$ also has minimal support.

Suppose for the sake of contradiction that every wall containing $p$ is in a hyperplane containing $p+v$.
In particular, the hyperplane of $R$ contains $p+v$, and the hypothesis on $v$ implies that $p$ is not in the relative interior of $R$ (the interior of $R$ relative to the hyperplane containing it).
Since $p\in R$, there exists $k$ such that $p$ is in a wall of $\D_k$.
Since $\D_k$ is finite, we can choose $\ep$ with $E>\ep>0$ small enough so that the $\ep$-ball $\B$ about $p$ does not intersect any wall of $\D_k$ not containing $p$.
Because $p$ is in a wall $\d$ of $\D_k$ contained in $R$, the ball $\B$ contains a point $q$ in $\relint(\d)$ not contained in any wall of $\D_k$ not parallel to $\d$ and not contained in the relative boundary of any wall of $\D_k$ parallel to $\d$.
Since $\D_k$ has minimal support, the product $f=\prod_{\d'}f_\d$ over all $\d'$ containing $q$ has $f\neq 1$.
Because $p$ is not in the relative interior of $R$, the ball $\B$ also contains a point $r\in H\setminus R$, where $H$ is the hyperplane containing $R$.
Furthermore, since every wall of $\D_k$ intersecting $\B$ is in a hyperplane containing $p+v$, we can choose $r$ such that the segment $\seg{qr}$ does not cross any walls not in $H$.
Now consider a loop $\gamma$ contained in $\B$ that passes through $H$ at $q$ and $r$ and that stays close enough to $\seg{qr}$ that it passes through no walls of $\D_k$ not contained in $H$.
There is exactly one non-identity contribution to $\p_{\gamma,\D_k}$, namely at $q$, so $\p_{\gamma,\D_k}$ is not the identity.
This contradicts the the consistency of $\D_k$ and thus (by Proposition~\ref{Dk consist}) the consistency of $\D$.
This contradiction proves that $p$ is contained in some wall of $\D$ whose hyperplane does not contain $p+v$.
\end{proof}

%

\begin{remark}\label{useless dimensions}
As mentioned in Remark~\ref{useless dimensions teaser}, we construct the scattering diagram in $V^*$ (the real span of the set $\set{e_i^*:i\in I_\uf}$), while \cite{GHKK} constructs the scattering diagram in a larger space, namely $M\otimes\reals$ (the real span of $\set{e_i^*:i\in I}$).
Suppose we were to construct the scattering diagram in $M\otimes\reals$ instead.
Each wall in a scattering diagram (here and in \cite{GHKK}) is contained in a hyperplane perpendicular to a vector in $N^+\subset N_\uf$, and each such hyperplane contains the real span of $\set{e_i^*:i\in I_\fr}$.
Using simple arguments similar to the proof of Lemma~\ref{bare point}, we see that, up to equivalence of scattering diagrams, we could take each wall in a \emph{consistent} scattering diagram to be the direct product of a cone in $V^*$ with the real span of $\set{e_i^*:i\in I_\fr}$.
Since we only care about consistent scattering diagrams, making the construction in $V^*$ is a matter of taking a quotient modulo some unnecessary dimensions.
\end{remark}

\begin{remark}\label{coeffs don't matter remark}
Once we choose the exchange matrix ${[\epsilon_{ij}]_{i,j\in I_\uf}}$, fix $V$, fix a basis ${(e_i:i\in I_\uf)}$ for $V$ and take $N_\uf$ to be the lattice generated by $\set{e_i:i\in I_\uf}$, the choice of a superlattice $N$ and the choice of an extension of $\set{\,\cdot\,,\,\cdot\,}$ to $N\times N$ corresponds to a choice of coefficients in the sense of cluster algebras of geometric type \cite[Section~2]{ca4}.
Thus the unnecessary dimensions mentioned in Remark~\ref{useless dimensions} are dual to the dimensions where the coefficients live.
Proposition~\ref{coeffs don't matter} shows that, because we left out the unnecessary dimensions, the choice of coefficients commutes with the construction of consistent scattering diagrams.
\end{remark}

\section{Scattering fans}\label{scatfan sec}
In the introduction, we asserted that a consistent scattering diagram ``cuts the ambient vector space into a complete fan.''
We now make that assertion precise and prove it.
We begin by recalling some basic definitions.

A \newword{convex cone} in a real vector space is a subset that is convex and closed under positive scaling and under addition.
(In fact, if the subset is convex and closed under positive scaling, then it is closed under addition.)
The \newword{relative interior} $\relint(C)$ of a convex set $C$ is its interior as a subset of its linear span.
A convex set is \newword{relatively open} if it is open as a subset of its linear span, or equivalently if it equals its relative interior.
The \newword{relative boundary} of a convex set is its boundary as a subset of its linear span.
A subset $F$ of a closed convex set $C$ is a \newword{face} if it is convex and has the property that any line segment $L$ contained in $C$ whose relative interior intersects $F$ has $L\subseteq F$.
(There is a better-known definition of faces of a closed convex \emph{polyhedron}, but we use this more general definition because we don't know that the cones we consider are polyhedral.)
A face $F$ of a closed convex cone $C$ is again a closed convex cone, and furthermore a subset of $F$ is a face of $F$ if and only if it is a face of $C$.
A \newword{fan} is a collection of closed convex cones that is closed under passing to faces and that has the property that, given any two cones $C$ and $D$ in the fan, the intersection $C\cap D$ is a face of $C$ and a face of $D$.
A fan is \newword{complete} if the union of its cones is the entire ambient vector space.

Now let $\D$ be a scattering diagram.
Recall that the \newword{rampart} of~$\D$ associated to $n_0\in N^+$ is the union of all walls of $\D$ contained in $n_0^\perp$.
Given $\D$ and $p\in V^*$, write $\Ram_{\D}(p)$ for the set of ramparts of $\D$ containing $p$ and write $\D\setminus\Ram_{\D}(p)$ for the set of walls of $\D$ \emph{not} contained in any rampart in $\Ram_{\D}(p)$.

Suppose $\D$ is a consistent scattering diagram with minimal support.
We define an equivalence relation on $V^*$ by declaring $p,q\in V^*$ to be \newword{$\D$-equivalent} if and only if there is a path $\gamma$ from $p$ to $q$ on which $\Ram_\D(\,\cdot\,)$ is constant.
In other words, ${p,q\in V^*}$ are $\D$-equivalent if and only if, first, $\Ram_\D(p)=\Ram_\D(q)$ and second, $p$ and $q$ are in the same path-connected component of $(\cap\Ram_\D(p))\setminus(\Supp(\D\setminus\Ram_\D(p)))$.
If $\Ram_\D(p)=\emptyset$, then by convention $\cap\Ram_\D(p)=V^*$, so points $p,q\in V^*\setminus\Supp(\D)$ are equivalent if and only if they are in the same path-connected component of $V^*\setminus\Supp(\D)$.
A \newword{$\D$-class} is a $\D$-equivalence class.

The closure of a $\D$-class is called a \newword{$\D$-cone}.
Define $\Fan(\D)$ to be the set of $\D$-cones together with their faces.
The $\D$-cones and $\Fan(\D)$ depend on $\D$ only up to equivalence as long as $\D$ has minimal support.  
We call $\Fan(\D)$ a \newword{scattering fan}.
Our goal in this section is to prove the following result.

\begin{theorem}\label{scat fan}
If $\D$ is a consistent scattering diagram with minimal support, then $\Fan(\D)$ is a complete fan in~$V^*$.
\end{theorem}

Some remarks are in order.

\begin{remark}\label{crazy fan}
Typically, a fan is required to have finitely many cones, but we allow infinite fans (fans with infinitely many cones).
The topology of a finite complete fan is spherical, in the sense that the fan induces a cell decomposition of the unit sphere about the origin.
Infinite fans, including $\Fan(\D)$, can be much more complicated.
For example, \cite[Example~1.15]{GHKK} describes a family of $2$-dimensional scattering diagrams---in fact, cluster scattering diagrams in the sense of Section~\ref{clus scat sec}.
(See also \cite[Section~3]{scatcomb}.)
When the associated scattering fan is infinite, its full-dimensional cones cover a proper subset of the plane---all but some cone $C$.
The cone $C$ is sometimes a ray, but it is typically $2$-dimensional.
Every ray contained in $C$ is a cone in the scattering fan.
Thus the topology of the fan is not a circle, but rather a line segment together with uncountably many isolated points.
\end{remark}

\begin{remark}\label{rampart simpler}
The point of using ramparts, instead of walls, to define $\Fan(\D)$ is that the geometry of ramparts is the geometry of the support of $\D$.
The exact wall structure is not invariant under equivalence, but Proposition~\ref{min sup} implies that the rampart structure is, for consistent scattering diagrams with minimal support. 
\end{remark}

\begin{remark}
The scattering fan depends only on the matrix ${[\epsilon_{ij}]_{i,j\in I_\uf}}$, in light of Proposition~\ref{coeffs don't matter}.
\end{remark}

Theorem~\ref{scat fan} says in particular that each $\D$-cone is indeed a closed convex cone.
We begin the proof by establishing that fact as part of the following proposition.

\begin{prop}\label{D convex}
If $\D$ is a consistent scattering diagram with minimal support, then each $\D$-class is a convex cone and each $\D$-cone is a closed convex cone.
\end{prop}

In preparation for the proof of Proposition~\ref{D convex}, we prove the following technical lemma.
Given a scattering diagram $\D$ and a linear subspace $U$ of~$V^*$, let $\D\setminus U$ be the set of walls $(\d,f_\d)$ in $\D$ such that the hyperplane containing~$\d$ does not also contain $U$. 
Even if $\D$ is consistent, there is no guarantee that $\D\setminus U$ is consistent, but the lemma asserts a version of consistency of $\D\setminus U$ for paths that lie in $U$.

\begin{prop}\label{consistent restrict}  
Suppose $\D$ is a consistent scattering diagram and suppose $U$ is a subspace of $ V^*$.
If a path $\gamma$ is contained in $U$ and generic in $\D\setminus U$, then the path-ordered product $\p_{\gamma,\D\setminus U}$ depends only on the endpoints of $\gamma$.
\end{prop}
\begin{proof}
Let $\H$ be the set of hyperplanes in $V^*$ that both contain $U$ and contain a wall in $\D$.
Let $p=\gamma(0)$ and $q=\gamma(1)$.
Let $v$ be any vector not contained in any hyperplane in $\H$.
In particular, $v\not\in U$, so that, for any $\ep>0$, the path $\gamma+\ep v$ is disjoint from $U$.
We can't guarantee that $\gamma+\ep v$ is generic for any particular $\ep>0$, so we can't necessarily define any path-ordered product $\p_{\gamma+\ep v,\D}$.

However, we claim that for $k\ge 1$, there exists $\ep_k>0$ such that for ${0<\ep\le\ep_k}$,
\begin{enumerate}
\item $\gamma+\ep v$ is generic in $\D_k$, and 
\item \label{seq} the sequence $(\d_1,f_{\d_1}),\ldots,(\d_\ell,f_{\d_\ell})$ of walls of $\D_k$ crossed by $\gamma+\ep v$ equals the sequence of walls of $\D_k\setminus U$ crossed by $\gamma$.
\end{enumerate}
Since $\gamma$ is generic in $\D\setminus U$, it is generic in $\D_k\setminus U$.
Since $\D_k$ is finite, we can choose $\ep_k$ small enough so that the sequence of walls of $\D_k\setminus U$ crossed by $\gamma+\ep v$ is constant for $0\le\ep\le\ep_k$, and by our choice of $v$, if we restrict to $0<\ep\le\ep_k$, then $\gamma+\ep v$ crosses no walls whose hyperplane contains $U$.
Thus we can satisfy \eqref{seq}.
Also because $\gamma+\ep v$ crosses no walls whose hyperplane contains $U$ for $0<\ep\le\ep_k$, to ensure that $\gamma+\ep v$ is generic in $\D_k$, we only need to consider how it crosses walls in $\D_k\setminus U$.
It is clear that we can choose $\ep_k$ small enough so that the endpoints of $\gamma+\ep v$ are not in $\Supp(\D_k\setminus U)$.
If $B$ is the relative boundary of a wall in $\D_k\setminus U$ or is the intersection of two walls of $\D_k\setminus U$, then $B$ is closed.
Since $B$ is disjoint from $\gamma$, we can choose $\ep_k$ small enough for $\gamma+\ep v$ to avoid $B$.
There are only finitely many such $B$, so we can avoid all of them, and we have proved the claim.

We consider the limit, as $k\to\infty$, of the compositions $\p_{\gamma+\ep_kv,\d_\ell}\circ\cdots\circ\p_{\gamma+\ep_kv,\d_1}$.
(The index $\ell$ and the walls $\d_1,\ldots,\d_\ell$ vary with $k$ in this limit.)
By the claim, this is the same limit encountered in the definition of $\p_{\gamma,\D\setminus U}$, so the limit is~$\p_{\gamma,\D\setminus U}$.

Since neither $p$ nor $q$ is contained in any wall of $\D\setminus U$, there is no positive lower bound on values of $\ep>0$ such that neither $p+\ep v$ nor $q+\ep v$ is in $\Supp(\D)$.
When neither $p+\ep v$ nor $q+\ep v$ is in $\Supp(\D)$, we have a well-defined function $\p_{p+\ep v,q+\ep v,\D}$.
For each $k\ge 1$, choose $\ep'_k$ with $0<\ep'_k\le\ep_k$ such that neither $p+\ep'_k v$ nor $q+\ep'_k v$ is in $\Supp(\D)$.
By the claim, the path $\gamma+\ep'_kv$ is generic in $\D_k$.
Choose points $p_0,\ldots,p_{\ell'}$ on $\gamma+\ep'_kv$ and in $V^*\setminus\Supp(\D)$ between all distinct intersection points of $\gamma+\ep'_kv$ with walls of $\D_k$, taking $p_0=p$ and $p_{\ell'}=q$.
To compute $\p_{p+\ep v,q+\ep v,\D}=\p_{p_{\ell'-1},p_{\ell'},\D}\circ\cdots\circ\p_{p_0,p_1,\D}$ for any $\ep\le\ep'_k$, we take in the $k\th$ step exactly the composition $\p_{\gamma+\ep'_kv,\d_\ell}\circ\cdots\circ\p_{\gamma+\ep'_kv,\d_1}$.
Thus for $\ep\le\ep'_k$, the functions $\p_{p+\ep v,q+\ep v,\D}$ and $\p_{\lim_{\ep\to0}\gamma+\ep v,\D}$ coincide modulo $\m^{k+1}$.
Therefore $\lim_{\ep\to0}\p_{p+\ep v,q+\ep v,\D}$ exists and equals the limit from the previous paragraph.

We have shown that $\p_{\gamma,\D\setminus U}=\lim_{\ep\to0}\p_{p+\ep v,q+\ep v,\D}$ (independent of the vector $v$ chosen).
Since $\lim_{\ep\to0}\p_{p+\ep v,q+\ep v,\D}$ depends only on $p$ and $q$, we are done.
\end{proof}

We now prove the first piece of Theorem~\ref{scat fan}.

\begin{proof}[Proof of Proposition~\ref{D convex}]
The assertion that each $\D$-cone is closed is by definition and is made here for emphasis.
Closures of convex sets are themselves convex---see, for example, \cite[Theorem~2.3.5]{Webster}---and closures of cones are cones.
Thus we need only show that every $\D$-class is a convex cone.
As mentioned above, this is equivalent to showing that every $\D$-class is convex and closed under positive scaling.
But closure under positive scaling is immediate, because applying a positive scaling to a point preserves the set of walls containing the point.
Thus our task is to prove that every $\D$-class is convex.
Since $\D$ is consistent, Proposition~\ref{Dk consist} says that $\D_k$ is consistent (modulo $\m^{k+1}$) for each $k\ge1$, and we will use that fact throughout the proof.
Proposition~\ref{min sup} says that up to equivalence and without changing the support of $\D$, we can assume that each $\D_k$ has minimal support.

Suppose $p$ and $q$ are $\D$-equivalent.
That is, $\Ram_\D(p)=\Ram_\D(q)$ and there exists a path $\tilde\gamma$ from $p$ to $q$ in $\cap\Ram_\D(p)$ that does not intersect any wall not in a rampart in $\Ram_\D(p)$.
Writing $U$ for the $\reals$-linear span of $\cap\Ram_\D(p)$, we see that $\tilde\gamma$ is contained in $U$ and generic in $\D\setminus U$.
Write $\gamma$ for the straight line segment $\seg{pq}$.
To complete the proof, we need to show that $\gamma$ is contained in the $\D$-equivalence class of $p$ and $q$, or in other words, that the points on $\gamma$ are all $\D$-equivalent.
We will prove this by establishing two claims that rule out certain types of intersections of $\gamma$ with walls of $\D$, and then showing that $\gamma$ never leaves $\Ram_\D(p)$.

We first claim that, for $k\ge1$, the segment $\gamma$ intersects no wall of $\D_k\setminus U$ transversely.
(We include, under the description ``intersecting a wall transversely'' the case where $\gamma$ passes through the relative boundary of the wall while passing transversely through the hyperplane of the wall.)
Suppose to the contrary that $\gamma$ intersects some wall of $\D_k\setminus U$ transversely.
By definition, every hyperplane containing a wall in $\D_k\setminus U$ intersects $U$ in a subspace of $U$ (of codimension $1$ in $U$).
In particular, it is possible to choose a vector $v$ and $\ep>0$ such that the line segment $\gamma+\ep v$ is contained in $U$ but is not contained in any hyperplane containing a wall in $\D_k\setminus U$.
Furthermore, it is possible to choose $v$ such that, for small enough $\ep>0$, the segment $\gamma+\ep v$ intersects some wall of $\D_k\setminus U$ transversely but is generic in $\D_k\setminus U$.
Possibly by making $\ep$ even smaller (but still positive), neither the segment from $p$ to $p+\ep v$ nor the segment from $q+\ep v$ to $q$ intersects any wall in $\D_k\setminus U$.
Let $\gamma^+$ be the path consisting of a segment from $p$ to $p+\ep v$ followed by the segment $\gamma+\ep v$ and then the segment from $q+\ep v$ to $q$.
This is generic in $\D_k\setminus U$.

The path $\tilde\gamma$ intersects no wall of $\D_k\setminus U$, so $\p_{\tilde\gamma,(\D_k)\setminus U}$ is the identity.
Since the initial and final segments of $\gamma^+$ intersect no walls in $\D_k\setminus U$ and the middle segment is a straight line segment crossing some wall of $\D_k\setminus U$, Lemma~\ref{tech lem} says that $\p_{\gamma^+,(\D_k)\setminus U}$ is not the identity.
By this contradiction to Proposition~\ref{consistent restrict}, we have established the claim that $\gamma$ crosses no wall of $\D_k\setminus U$ transversely.

We next claim that the segment $\gamma$ does not intersect any walls of $\D\setminus\Ram_\D(p)$.
Since $\Ram_\D(p)=\Ram_\D(q)$, any hyperplane containing $U=\Span_\reals(\cap\Ram_\D(p))$ also contains $\gamma$, and thus cannot be crossed transversely by $\gamma$.
Thus if $\gamma$ intersects a wall of $\D\setminus\Ram_\D(p)$ transversely, then $\gamma$ intersects a wall of $\D\setminus U$ transversely.
The latter is ruled out by the first claim, so it remains to show a contradiction in the case where $\gamma$ intersects some wall of ${\D\setminus\Ram_\D(p)}$ non-transversely.
(We include, under the description ``intersecting a wall non-transversely'' the case where $\gamma$ passes only through the relative boundary of the wall while passing non-transversely through the hyperplane of the wall.)

Suppose to the contrary that $\gamma$ intersects some wall of $\left(\D\setminus\Ram_\D(p)\right)\cap \D_k$ non-transversely.
Let $r$ be the first point where $\gamma$ intersects a wall of $\left(\D\setminus\Ram_\D(p)\right)\cap\D_k$ non-transversely.
Let $R$ be a rampart of $\D_k$ containing a wall of $\left(\D\setminus\Ram_\D(p)\right)\cap\D_k$ intersected non-transversely at $r$.
For small enough $\ep>0$, the $\ep$-ball $\B$ about $r$ does not intersect any wall of $\D_k$ not containing $r$.
The rampart $R$ is a union of (finitely many) walls of $\D_k$, each of which is a closed polyhedral cone of codimension $1$.
Its boundary is a union of (finitely many) polyhedral cones of codimension $2$, each contained in a maximal proper face of a wall contained in $R$.
The point $r$ is in the boundary of $R$, and is thus contained in some cone $C$ of codimension $2$ that is part of the boundary of $R$.
Furthermore, $C$ can be chosen so that $\Span_\reals(C)$ does not contain $\gamma$.
Choose a point $r'$ in $C\cap\B$ that is not contained in any rational hyperplane not containing $C$.
Lemma~\ref{bare point} says that $r'$ is contained in some wall $(\d,f_\d)$ of $\D_k$ not parallel to $R$.
By the definition of $\ep$, the wall $\d$ contains $r$.
The hyperplane containing $\d$ contains $C$, so it cannot also contain $\gamma$, because it is not parallel to $R$.
We have shown that the segment $\gamma$ intersects $\d$ transversely.
Since the hyperplane containing $\d$ doesn't contain $\gamma$, in particular it doesn't contain $U$, so $\d$ is in $\D_k\setminus U$.
However, we already proved the first claim (that $\gamma$ does not intersect any wall of $\D_k\setminus U$ transversely).
By this contradiction, we have completed the proof of the second claim (that $\gamma$ does not intersect any walls of $\D\setminus\Ram_\D(p)$).

We now show that $\gamma$ is contained in $\cap\Ram_\D(p)$.
In fact, we show that, for every $k\ge1$, if a rampart of $\D_k$ contains $p$, then it contains $\gamma$.
(This is enough, because each rampart of $\D$ is a union of ramparts of $\D_k$, as $k\to\infty$.)
Suppose $k\ge1$ and $R$ is a rampart of $\D_k$ with $p\in R$.
For the sake of contradiction, suppose $\gamma\not\subseteq R$.
Since $R$ is a finite union of walls, there is some point $r$ where $\gamma$ leaves $R$ by crossing the boundary of $R$.
As in the proof of the second claim above, $r$ is contained in some cone $C$ of codimension $2$ that is part of the boundary of $R$ and $C$ can be chosen so that $\Span_\reals(C)$ does not contain~$\gamma$.
Arguing as in the conclusion of that proof, we exhibit a wall $\d\in\D\setminus U$ such that $\gamma$ intersects $\d$ (transversely).
But this contradicts the first claim, and we conclude that $\gamma\subseteq R$.

We have shown that $\gamma$ is contained in $\cap\Ram_\D(p)$ and does not intersect any walls in $\D\setminus\Ram_\D(p)$.
Thus $\Ram_\D(\,\cdot\,)$ is constant on $\gamma$, and thus by definition, the points of $\gamma$ are all $\D$-equivalent.
\end{proof}


The next step in the proof of Theorem~\ref{scat fan} is to reduce the theorem to proving that $\Fan(\D_k)$ is a fan for each $k\ge1$.
We will need the following well-known fact about convex closed sets.

\begin{prop}\label{cone definite}  
Suppose $C$ is a closed convex set and suppose $p\in\relint(C)$.
Then $C$ is the set of points $q$ such that there exists a path from $p$ to $q$ contained in $\relint(C)$ except possibly at $q$.
Given $q\in C$, this path can be taken to be a straight line segment.
\end{prop}
\begin{proof}
If there exists a path from $p$ to $q$ contained in $\relint(C)$ except possibly at $q$, then in particular $q$ is the limit of the points in the path, so $q$ is in the closure of $\relint(C)$.
It is well-known that a closed convex set is the closure of its relative interior (see, for example \cite[Theorem~2.3.8]{Webster}), so $q\in C$.

Conversely, suppose $q\in C$ and consider the line segment $\seg{pq}$.
A well-known fact on convexity says that, given a segment $\gamma$ connecting a point $p$ in the relative interior of a convex set to a point $q$ in the closure of that set, the segment is contained in the relative interior, except possibly $q$.
(See for example \cite[Theorem~2.3.4]{Webster}.)
\end{proof}

We continue by relating $\D$-equivalence to $\D_k$ equivalence.

\begin{prop}\label{D Dk equiv}
Suppose that $\D$ is consistent and has minimal support and that $\D_k$ has minimal support (as a scattering diagram modulo $\m^{k+1}$) for each $k\ge 1$.
Two points $p,q\in V^*$ are $\D$-equivalent if and only if they are $\D_k$-equivalent for all $k\ge 1$.
\end{prop}
\begin{proof}
Suppose $p$ and $q$ are $\D_k$-equivalent for all $k\ge 1$.
Then for all $k\ge 1$, by Proposition~\ref{D convex}, $\Ram_{\D_k}(\,\cdot\,)$ is constant on the line segment $\seg{pq}$.
Since each rampart of $\D$ is a union, as $k\to\infty$, of ramparts of $\D_k$, also $\Ram_\D(\,\cdot\,)$ is constant on $\seg{pq}$.

Conversely, suppose $p$ and $q$ are $\D$-equivalent and fix $k\ge1$.
If $p=q$, then we are done, so assume not.
Then by Proposition~\ref{D convex}, $\Ram_\D(\,\cdot\,)$ is constant on the line segment $\seg{pq}$, and in particular $\seg{pq}$ does not intersect any wall of $\D\setminus\Ram_\D(p)$.
Thus $\seg{pq}$ does not intersect any wall of $\D_k\setminus\Ram_\D(p)$.
If $\Ram_{\D_k}(\,\cdot\,)$ is not constant on $\gamma$, then Lemma~\ref{bare point} implies that $\seg{pq}$ intersects a wall $\d$ of $\D_k$ transversely.
But then the rampart containing $\d$ intersects $\seg{pq}$ at only one point, contradicting the fact that $\Ram_\D(\,\cdot\,)$ is constant on the line segment $\seg{pq}$.
\end{proof}

We next prove some general facts about collections of cones defined by equivalence relations.
Given an equivalence relation $\equiv$ all of whose classes are convex cones, define an \newword{$\equiv$-cone} to be the closure of an $\equiv$-class.
Let $\F_\equiv$ be the collection of all $\equiv$-cones and faces of $\equiv$-cones.

\begin{prop}\label{relint class equiv}
If $\equiv$ is an equivalence relation all of whose classes are convex cones, then every $\equiv$-cone $C$ is the closure of a unique $\equiv$-class, and this class contains $\relint(C)$.
\end{prop}
\begin{proof}
The cone $C$ is the closure of some $\equiv$-class $C'$, which is a convex cone.
It is well-known that the relative interior of a convex set $A$ equals the relative interior of the closure of $A$.
(See for example \cite[Theorem~2.3.8]{Webster}.)
Thus $\relint(C)=\relint(C')\subseteq C'$.
Now $C$ can't be the closure of a different $\equiv$-class, because if so, both classes contain $\relint(C)\neq\emptyset$.
\end{proof}

\begin{prop}\label{lim fan}
Suppose $(\,\equiv_k\,:\,k\ge1)$ is a family of equivalence relations on on a real vector space such that the equivalence classes are relatively open convex cones.
Suppose also that $p\equiv_\ell q\implies p\equiv_kq$ for all $k$ and $\ell$ with $1\le k\le\ell$.
Define $\equiv$ to be the equivalence relation with $p\equiv q$ if and only if $p\equiv_kq$ for all $k\ge1$.
If $\F_{\equiv_k}$ is a fan for all $k\ge1$, then $\F_\equiv$ is a fan, and is the coarsest common refinement of the $\F_{\equiv_k}$.
\end{prop}
\begin{proof}
For any $p$, the $\equiv$-class of $p$ is the intersection of the $\equiv_k$-classes of $p$, and thus is convex.
To check that a collection of cones is a fan, it is enough to check that it is closed under passing to faces and for \emph{maximal} cones $C$ and $D$, the intersection $C\cap D$ is a face of each.
(This is well-known.  See for example \cite[Lemma~5.19]{universal} for a proof.)  
The collection $\Fan_\equiv$ is closed under passing to faces, by definition.
Since each maximal cone of $\Fan_\equiv$ is the closure of a $\equiv$-class, to show that $\F_\equiv$ is a fan, it is enough to check that for any closures $C$ and $D$ of $\equiv$-classes, the intersection $C\cap D$ is a face of $C$.
Appealing to Proposition~\ref{relint class equiv}, let $C'$ be the unique $\equiv$-class whose closure is $C$.
For each $k\ge1$, let $C'_k$ be the $\equiv_k$-class containing $C'$ and let $C_k$ be the closure of $C'_k$.
Define $D'$, $D'_k$, and $D_k$ analogously.

Suppose $L$ is a line segment contained in $C$ and whose relative interior intersects~$D$.
We need to show that $L\subseteq D$.
For each $k\ge 1$, since $C\subseteq C_k$ and $D\subseteq D_k$, the segment $L$ is contained in $C_k$ and its interior intersects $D_k$.
Since $\F_{\equiv_k}$ is a fan, $C_k\cap D_k$ is a face of $D_k$, so $L$ is contained in $C_k\cap D_k$.
Thus $L$ is contained in each~$D_k$.

To show that $L$ is contained in $D$, we prove that $D=\bigcap_{k\ge1}D_k$.
The containment $D\subseteq\bigcap_{k\ge1}D_k$ is immediate.
Let $q$ be a point in $\bigcap_{k\ge1}D_k$ and let $p\in\relint(D)\subseteq D'$. Then $p\in D'_k$ for all $k\ge1$.
It is well-known that a relatively open convex set is the relative interior of its closure.
(This follows from \cite[Theorem~2.3.8]{Webster}, already quoted above.)
Thus since each $D'_k$ is relatively open and convex, $D'_k=\relint(D_k)$, so $p\in\relint(D_k)$ for all $k\ge1$.  
Now Proposition~\ref{cone definite} says that each point in $\seg{pq}$, except possibly $q$, is in $\relint(D_k)=D'_k$.
Thus every point in $\seg{pq}$, except possibly $q$, is in $\bigcap_{k\ge1} D'_k=D'$.
Thus $q\in D$.

We have proved that $L\subset C\cap D$ and thus that $C\cap D$ is a face of $C$.
Therefore $\F_\equiv$ is a fan.
The assertion about coarsest common refinement is immediate.
\end{proof}

The following special case of Proposition~\ref{relint class equiv} is used often enough to warrant stating separately.
\begin{prop}\label{relint class}
Suppose $\D$ is a consistent scattering diagram with minimal support.
Every $\D$-cone $C$ is the closure of a unique $\D$-class.
This class contains $\relint(C)$.
\end{prop}

The final tool needed for the proof of Theorem~\ref{scat fan} is the following proposition, which will take some effort.

\begin{prop}\label{seg int}
Suppose $\D$ is a consistent scattering diagram with minimal support.
If $C$ is a $\D$-cone and $p,q\in C$, then the relative interior of the line segment $\seg{pq}$ is contained in a $\D$-class.
\end{prop}
\begin{proof}
If $p=q$, then the assertion is trivial, so assume $p\neq q$.
Proposition~\ref{relint class} says that $C$ is the closure of a unique $\D$-class $C'\supseteq\relint(C)$.
Proposition~\ref{D convex} says that $\seg{pq}$ is contained in $C$.
We will show that $\relint(\seg{pq})$ is contained in some $\D$-class, but this class may not be~$C'$.

Let $r$ be any point in $\relint(C)$.
If $r$ is on the line containing $p$ and $q$, then the assertion follows immediately from Propositions~\ref{cone definite} and~\ref{relint class}, so assume not.
Proposition~\ref{cone definite} implies that the entire triangle $T$ with vertices $p$, $q$, and $r$, except possibly the segment $\seg{pq}$, is contained in $\relint(C)$ and therefore in $C'$.

Let $s_1$ and $s_2$ be distinct points in $\relint(\seg{pq})$ such that for $i\in\set{1,2}$, 
\begin{enumerate}[(i)]
\item \label{si rat}
Every rational hyperplane containing $s_i$ also contains $\seg{pq}$. 
\item \label{si int}
No two rational hyperplanes have their intersection contained in the hyperplane containing $s_i$ and orthogonal to $\seg{pq}$.
\end{enumerate}
This is possible because the two requirements on $s_i$ force $s_i$ to avoid a countable number of points on $\seg{pq}$.  
Thus the possibilities for $s_i$ form a dense subset of $\seg{pq}$.

We first claim that the line segment $\seg{s_1s_2}$ intersects no wall of $\D$ transversely.
Most of the proof consists of checking this claim.
(As before, we include, under the description ``intersecting a wall transversely'' the case where $\seg{s_1s_2}$ passes through the relative boundary of the wall while passing transversely through the hyperplane of the wall.)
It is enough to prove the claim for each $\D_k$.

Fix $k\ge1$.
Since $\D_k$ is finite and by \eqref{si rat}, there exists $\ep>0$ such that, for $i\in\set{1,2}$, the $\ep$-ball about $s_i$ does not intersect any wall of $\D_k$ whose hyperplane does not contain $\seg{pq}$.
Choose points $t_1$ and $t_2$ in the triangle $T$ with $t_1-s_1$ and $t_2-s_2$ both orthogonal to $\seg{pq}$ and with $|t_1-s_1|=|t_2-s_2|<\ep$.
(Here, $|t_i-s_i|$ is the usual Euclidean distance.)
In particular, $t_1$ and $t_2$ are in $C'$ and thus are $\D$-equivalent, so Proposition~\ref{D convex} says that the line segment $\seg{t_1t_2}$ is contained in $C'$. 
By Proposition~\ref{D Dk equiv}, $\Ram_{\D_k}(\,\cdot\,)$ is constant on $C'$, so every wall of $\D$ intersecting $\seg{t_1t_2}$ is contained in a rampart of $\D_k$ that contains $C'$, and thus $C$, and thus $\seg{pq}$.
There exists $\ep'$ such that the $\ep'$-neighborhood of $\seg{t_1t_2}$ doesn't intersect any wall of $\D_k$ whose hyperplane does not contain $\seg{pq}$.
Within the $\ep'$-neighborhood of $\seg{t_1t_2}$, we can choose points $u_1$ and $u_2$ such that 
\begin{enumerate}[(i)]\setcounter{enumi}{2}
\item \label{us rect}
$u_1$, $u_2$, $s_2$ and $s_1$ form a rectangle with $|u_1-s_1|=|u_2-s_2|<\ep$,
\item $\seg{u_1u_2}$ is not contained in any rational hyperplane containing $\seg{pq}$, 
\end{enumerate}
Now, by construction, 
\begin{enumerate}[(i)] \setcounter{enumi}{4}
\item \label{u free}
the line segment $\seg{u_1u_2}$ does not intersect any walls of $\D_k$.
\end{enumerate}

Now suppose for the sake of contradiction that $\seg{s_1s_2}$ intersects a wall of $\D_k$ transversely.
We choose points $v_1$ and $v_2$ such that 
\begin{enumerate}[(i)] \setcounter{enumi}{5}
\item \label{rz rect}
$v_1$, $v_2$, $s_2$ and $s_1$ form a rectangle with $|v_1-s_1|=|v_2-s_2|<\ep$,
\item \label{r no hyp}
$\seg{v_1v_2}$ is not contained in any rational hyperplane containing $\seg{pq}$, 
\item \label{r int}
$\seg{v_1v_2}$ intersects a wall of $\D_k$ transversely,
\item \label{r gen}
$\seg{v_1v_2}$ is a generic path for $\D_k$, and 
\item \label{qr gen}
The segments $\seg{u_1v_1}$ and $\seg{u_2v_2}$ are generic paths for $\D_k$, and neither intersects any wall of $\D_k$ whose hyperplane does not contain $\seg{pq}$.

\end{enumerate}
To see why this is possible, consider:
Satisfying \eqref{rz rect} and \eqref{r no hyp} amounts to choosing a direction orthogonal to $\seg{pq}$ that avoids rational hyperplanes containing $\seg{pq}$ and then choosing a length $|v_1-s_1|=|v_2-s_2|<\ep$.
Since $\seg{s_1s_2}$ intersects a wall of $\D_k$ transversely but $s_1$ and $s_2$ are not contained in any wall of $\D_k$ whose hyperplane does not contain $\seg{pq}$, we see that the \emph{relative interior} of $\seg{s_1s_2}$ intersects a wall of $\D_k$ transversely.
Thus if the length is chosen small enough, the direction can be chosen so as to satisfy \eqref{r int}.
For $i\in\set{1,2}$, since $u_i$ and $v_i$ are in the $\ep$-ball about $s_i$, the segment $\seg{u_iv_i}$ does not intersect any walls of $\D_k$ whose hyperplane does not contain $\seg{pq}$.
Thus $\seg{u_iv_i}$ can be made generic by avoiding all pairwise intersections $\d\cap\d'$ of nonparallel walls of $\D_k$ containing $\seg{pq}$ and all boundaries $\partial\d$ of walls of $\D_k$ containing $\seg{pq}$.
Any such $\d\cap\d'$ contains $\seg{pq}$ and is of codimension at least $2$, so the intersection of $\d\cap\d'$ with the hyperplane orthogonal to $\seg{pq}$ and containing $s_i$ is of codimension at least $3$.
By \eqref{si int}, any such $\partial\d$, intersected with the hyperplane orthogonal to $\seg{pq}$ at $s_i$, yields a set of codimension at least $3$.
Thus the set of points $v_i$ such that $\seg{u_iv_i}$ intersects such a $\d\cap\d'$ and/or such a $\partial\d$ has codimension at least $2$.
Also, $\seg{v_1v_2}$ is generic if it misses all intersections of walls and all boundaries of walls.
This amounts to choosing the $v_i$ avoid a finite collection of sets of codimension at least $2$.
We can vary the $v_i$ within a set of codimension $1$, so we can satisfy \eqref{r gen} and \eqref{qr gen} by changing the chosen direction and length slightly without losing \eqref{rz rect}, \eqref{r no hyp}, and \eqref{r int}.

Let $U$ be the intersection of all hyperplanes $n^\perp\subset V^*$ containing $\seg{pq}$ for $n\in N^+$.
(Possibly there are no such hyperplanes, in which case $U=V^*$.)
By \eqref{r int}, the segment $\seg{v_1v_2}$ crosses a wall $(\d,f_\d)$ transversely.  
Since $\seg{v_1v_2}$ is parallel to $\seg{pq}$, the wall $\d$ does not contain $\seg{pq}$, and therefore, because $U$ contains $\seg{pq}$, the hyperplane $H$ containing $\d$ does not contain $U$.
Since $U$ is rational, there exists a nonzero $m\in M^\circ\cap (U\setminus H)$.
Let $\gamma'$ be the path obtained by concatenating the segments $\seg{v_1u_1}$, $\seg{u_1u_2}$, and $\seg{u_2v_2}$.
By \eqref{qr gen}, because $m\in U$, and since $\seg{u_1u_2}$ does not intersect any walls of $\D_k$, we have $\p_{\gamma',\D_k}(z^m)=z^m$.
On the other hand, by \eqref{r gen}, the segment $\seg{v_1v_2}$ satisfies the hypothesis of Lemma~\ref{tech lem detailed}.
That lemma asserts that $\p_{\seg{v_1v_2},\D_k}(z^m)\neq z^m$.
We have contradicted the consistency of $\D_k$, and thus, by Proposition~\ref{Dk consist}, contradicted the consistency of $\D$.
This contradiction proves that $\seg{s_1s_2}$ intersects no wall of $\D_k$ transversely and thus proves the claim that $\seg{s_1s_2}$ intersects no wall of $\D$ transversely.

We next claim that $\Ram_\D(s)$ is constant for $s\in\seg{s_1s_2}$.
Suppose to the contrary.
Then there exists $k\ge 1$ such that $\Ram_{\D_k}(s)$ is not constant for $s\in\seg{s_1s_2}$.
By the first claim, the variation in the function $\Ram_{\D_k}(s)$ involves only ramparts whose hyperplane contains $\seg{s_1s_2}$.
Ramparts in $\D_k$ are closed, so up to swapping $s_1$ and $s_2$ there is a rampart $R$ and a point $s\in R\cap\seg{s_1s_2}$ such that $s+\ep (s_1-s_2)\not\in R$ for all small enough $\ep>0$.  
Just as in the proof of Proposition~\ref{D convex}, we use Lemma~\ref{bare point} to conclude that $\seg{s_1s_2}$ intersects a wall of $\D_k$ transversely at $s$.
This contradicts the first claim, and we have proved the claim that $\Ram_\D(s)$ is constant for $s\in\seg{s_1s_2}$.

This second claim says that, by definition, $s_1$ and $s_2$ are $\D$-equivalent.
This is true for any two points chosen from a dense subset of $\seg{pq}$.
Since $\D$-classes are convex by Proposition~\ref{D convex}, the relative interior of $\seg{pq}$ is contained in a $\D$-class.
\end{proof}

We now complete the proof of the main result of this section: When $\D$ is a consistent scattering diagram with minimal support, $\Fan(\D)$ is a complete fan.

\begin{proof}[Proof of Theorem~\ref{scat fan}]
Completeness is immediate because each element of $V^*$ is in some $\D$-class and thus in some $\D$-cone.
By Lemma~\ref{min sup}, we can assume, up to equivalence of scattering diagrams without changing the $\D$-equivalence relation, that each $\D_k$ has minimal support.
Proposition~\ref{D Dk equiv} says that two points are $\D$-equivalent if and only if they are $\D_k$-equivalent for all $k\ge1$.
If $\ell>k\ge1$, then $(\D_\ell)_k=\D_k$, so Proposition~\ref{D Dk equiv} also says that if two points are $\D_\ell$-equivalent then they are $\D_k$-equivalent.
Thus, in light of Proposition~\ref{lim fan}, we only need to prove, for a \emph{finite} consistent scattering diagram~$\D$ (possibly working modulo $\m^{k+1}$ for some $k\ge1$) with minimal support, that each $\D$-class is a relatively open convex cone and that $\Fan(\D)$ is a fan.

Proposition~\ref{D convex} says that $\D$-classes are convex cones, and we now check relative openness.
We can obtain any $\D$-class by choosing a set $\R$ of ramparts, taking the intersection $\cap\R$, removing all points contained in ramparts not in $\R$ and choosing a component of what remains.
Since $\D$ is finite, $\cap\R$ is an intersection of finite unions of closed sets, and thus closed.
Lemma~\ref{bare point} implies that the every point in the relative boundary of $\cap\R$ is removed.
(Given $p$ in the relative boundary, there exists $v$ in the linear span of $\cap\R$ such that $p+\ep v\not\in \cap R$ for all small enough $\ep>0$.
This $p$ is also in the relative boundary of some rampart $R$ in $\R$. 
Apply Lemma~\ref{bare point} to this $R$, $p$, and $v$.)
Since we are removing a finite number of walls, all of which are closed, each path-connected component is relatively open, as desired.

We verified in Proposition~\ref{D convex} that each element of $\Fan(\D)$ is a convex cone, and $\Fan(\D)$ is closed under passing to faces, by construction.
As explained in the proof of Proposition~\ref{lim fan}, to complete the proof that $\Fan(\D)$ is a fan, it is enough to check the intersections of \emph{maximal} cones.
Each maximal cone is a $\D$-cone, so we check that for any two $\D$-cones $C_1$ and $C_2$, the intersection $C_1\cap C_2$ is a face of $C_1$.

Suppose $L$ is a line segment contained in $C_1$ with $\relint(L)\cap C_2\neq\emptyset$.
We need to show that $L\subseteq C_2$, but since $C_2$ is closed, $\relint(L)\subseteq C_2$ is enough.

Proposition~\ref{seg int} says that $\relint(L)$ is contained in a $\D$-class, so that $\Ram_\D(\,\cdot\,)$ is constant on $\relint(L)$.
Let $q$ be any point in $C_2\cap\relint(L)$ and let $p$ be a point in $\relint(C_2)$.
Appealing to Proposition~\ref{relint class}, let $C_2'$ be the unique $\D$-class whose closure is $C_2$.
By Proposition~\ref{cone definite}, the segment $\seg{pq}$, except possibly the point $q$, is contained in $\relint(C_2)$ and thus by Proposition~\ref{relint class}, contained in $C_2'$.

Let $r$ be any point in $\relint(L)$.
Since $q$ and $r$ are in $\relint(L)$, since $\Ram_\D(\,\cdot\,)$ is constant on $\relint(L)$, and since $\D$ is finite, there exists $\ep>0$ such that every wall intersecting the $\ep$-neighborhood of $\relint(\seg{qr})$ is contained in a rampart containing $\seg{qr}$.
Since the segment $\seg{pq}$, except possibly $q$, is contained in $C_2'$, we know that $\Ram_\D(\,\cdot\,)$ is constant on $\seg{pq}$ except possibly at $q$.

Now consider a differentiable path $\gamma$ from $p$ to $r$ that follows $\seg{pq}$ until it comes within $\ep$ of the point $q$ and then goes to $r$, staying within the triangle with vertices $p$, $q$, and $r$ and staying within $\ep$ of $\seg{qr}$.
If $\Ram_\D(\,\cdot\,)$ is not constant on $\gamma$, then any changes in $\Ram_\D(\,\cdot\,)$ occur after $\gamma$ diverges from $\seg{pq}$.
But after $\gamma$ diverges from $\seg{pq}$, it is within $\ep$ of $\relint(\seg{qr})$, so all ramparts containing points in $\gamma$ also contain $\seg{qr}$.
If $\gamma$ enters or leaves a rampart (except possibly at $r$), Lemma~\ref{bare point} implies that it crosses a wall transversely.
But no wall whose rampart contains $\seg{qr}$ and a point on $\gamma$ (except possibly $r$) can intersect $\gamma$ transversely, because $\gamma$ is in the triangle with vertices $p$, $q$, and $r$.
We see that $\Ram_\D(\,\cdot\,)$ is constant on $\gamma$, except possibly at $r$.
We conclude that $\gamma$ (except possibly the point $r$) is in $C'_2$, so $r$ is in $C_2$.
This was true for arbitrary $r\in\relint(L)$, so $\relint(L)\subseteq C_2$.
\end{proof}

\section{Cluster scattering diagrams}\label{clus scat sec}
In this section, we quote from \cite{GHKK} the definition of cluster scattering diagrams and recast a result from \cite{GHKK} on mutation of cluster scattering diagrams.
We then review the definition of mutation fans and prove that cluster scattering fans refine mutation fans.

The initial data for a cluster scattering diagram, described in Section~\ref{scat sec}, essentially amounts to the choice of $I$, $I_\uf$, a basis $\s=(e_i:i\in I_\uf)$ for a real vector space $V$ and the matrix ${\tB=[\epsilon_{ij}]_{i\in I_\uf,j\in I}}$.
(We say ``essentially'' here because the quantities $\epsilon_{ij}$ are not determined for $i,j\in I_\fr$.
However, these quantities are irrelevant for our purposes.) 
The requirements on $\tB$ are that it must have integer entries, that it must have rank $|I_\uf|$ and that there must exist positive integers $(d_i:i\in I_\uf)$ with $d_i\epsilon_{ij}=-d_j\epsilon_{ji}$ for all $i,j\in I_\uf$.
If the $d_i$ exist, then they are uniquely determined by $\tB$ and the requirement that $\gcd_{i\in I}(d_i)=1$.
The square submatrix ${B=[\epsilon_{ij}]_{i,j\in I_\uf}}$ is called an \newword{exchange matrix} and $\tB$ is called an \newword{extended exchange matrix}.

The \newword{cluster scattering diagram} associated to the extended exchange matrix $\tB$ and the basis $\s$ is the scattering diagram $\Scat(\tB,\s)$ whose existence and uniqueness (up to equivalence) is guaranteed in the following theorem \cite[Theorem~1.12]{GHKK}.

\begin{theorem}\label{key scat}  
There exists a consistent scattering diagram $\Scat(\tB,\s)$ containing $\set{(e_i^\perp,1+\zeta_i):i\in I_\uf}$ such that $\Scat(\tB,\s)\setminus\set{(e_i^\perp,1+\zeta_i):i\in I_\uf}$ consists only of outgoing walls.
These conditions uniquely define $\Scat(\tB,\s)$ up to equivalence.
\end{theorem}

Proposition~\ref{coeffs don't matter}, in the language of exchange matrices and extended exchange matrices, says that as long as $\tB$ has full rank $|I_\uf|$, the structure of the scattering diagram depends only on $B$.
More specifically, if two extended exchange matrices both extend the same exchange matrix, the scattering diagram of one can be obtained from the scattering diagram of the other by a global change of variables in all of the power series $f_\d$.
In particular, assuming that $\Scat(\tB,\s)$ is taken with minimal support, the scattering fan $\Fan(\Scat(\tB,\s))$ depends only on $B$ and $\s$.
We assume a standard choice of $\s$ and write $\ScatFan(B)$ for $\Fan(\Scat(\tB,\s))$.
This is the \newword{cluster scattering fan} for $B$.

\subsection{Mutation of cluster scattering diagrams}\label{mut sec}
The crucial operation on exchange matrices is \newword{mutation}.
We define mutation on any matrix $[b_{ij}]_{i\in R,j\in C}$ such that $I_\uf\subseteq R$ and $I_\uf\subseteq C$.
For any $k\in I_\uf$, define $\mu_k([b_{ij}])$ to be the matrix $[b'_{ij}]_{i\in R,j\in C}$ with
\begin{equation}\label{b mut}
b_{ij}'=\left\lbrace\!\!\begin{array}{ll}
-b_{ij}&\mbox{if }i=k\mbox{ or }j=k;\\
b_{ij}+\sgn(b_{kj})\,[b_{ik}b_{kj}]_+&\mbox{otherwise.}
\end{array}\right.
\end{equation}
Here $\sgn(b)=\frac{b}{|b|}$ if $b\neq0$, or $\sgn(b)=0$ if $b=0$.
Also, $[a]_+$ means $\max(a,0)$.

The map $\mu_k$ is an involution.
Furthermore, when we apply $\mu_k$ to $\tB$ to obtain $\mu_k(\tB)=[\ep'_{ij}]_{i\in I_\uf,j\in I}$, the integers $(d_i:i\in I_\uf)$ still have the property that $d_i\ep'_{ij}=-d_j\ep'_{ji}$ for all $i,j\in I_\uf$.
Given a finite sequence $\ks=k_q,\ldots,k_1$ of indices in~$I_\uf$, we define $\mu_\ks$ to be $\mu_{k_q}\circ\mu_{k_{q-1}}\circ\cdots\circ\mu_{k_1}$.
%

A natural and crucial question is:  What happens to the cluster scattering diagram when we mutate $\tB$?
But there are at least two ways to interpret the question.
One interpretation assumes that we mutate $\tB$ but fix all of the initial data that does not depend on the basis $\s$ (the ``fixed data'').
That is, when we mutate at index $k$, we replace $\s$ by a new basis $\mu_k(\s)$ such that $(\mu_k(\tB),\mu_k(\s))$ defines the same lattice $N$ and pairing $\set{\,\cdot\,,\,\cdot\,}$ as $(\tB,\s)$.
Another interpretation assumes that we mutate $\tB$ and fix the basis $\s$, thus still fixing the lattice $N$ but now changing the form $\set{\,\cdot\,,\,\cdot\,}$.

The question was answered, under the first interpretation, as \cite[Theorem~1.24]{GHKK}, and we will show below how \cite[Theorem~1.24]{GHKK} leads to an answer to the question under the second interpretation.
(In the skew-symmetric case, this was already done as \cite[Lemma~5.2.1]{Muller} and in the rank-$2$ case it was done in \cite[Section~3]{CGMMRSW}.)
We are interested in the second interpretation for two reasons.
First, we want to think of the scattering diagram as depending only on $\tB$, so we take away the choice of basis by fixing $\s$ once and for all.
Second, the transformation from $\Scat(\tB,\s)$ to $\Scat(\mu_k(\tB),\s)$ is given by the piecewise linear map identified in \mbox{\cite[(7.18)]{ca4}} as describing initial-seed mutation of $\g$-vectors and used in \cite{universal} to understand universal geometric coefficients.
In particular, we point out in \cite[Section~2.3]{scatcomb} that the answer under the second interpretation proves one of the main conjectures of \cite{ca4}, namely \cite[Conjecture~7.12]{ca4}.
(A global transpose is necessary to make the connection to $\g$-vectors, as explained in \cite[Section~2]{scatcomb}.)

We now describe the operation that takes $\Scat(\tB,\s)$ to $\Scat(\mu_k(\tB),\s)$.
We have borrowed notational conventions from \cite{NZ} and have borrowed elements of the description of the map from \cite[Lemma~5.2.1]{Muller}.

For each $k\in I_\uf$, define $J_k$ to be the square matrix, indexed by $I_\uf$, that agrees with the identity matrix except that the $kk$ entry is $-1$.
For any matrix $A$, define $A^{k\bullet}$ to be the matrix that agrees with $A$ in row $k$ and has zeroes everywhere else.
We will use these matrices to define some linear maps on $V$ and on $V^*$.
We treat vectors in $V$ as column vectors giving coordinates with respect to the $e_i$, so that matrices act on $V$ from the left, and we treat vectors in $V^*$ as row vectors giving coordinates with respect to the $f_i$, so that matrices act on $V^*$ from the right.

Since we want a map on \emph{equivalence classes} of scattering diagrams, we can assume up to equivalence that the hyperplane $e_k^\perp$ does not intersect the relative interior of a wall except when that wall is contained in $e_k^\perp$.
(If the relative interior of some wall $(\d,f_\d)$ intersects $e_k^\perp$, replace that wall by two walls $(\d_+,f_\d)$ and $(\d_-,f_\d)$, where $\d_\pm=\set{v\in\d:\pm\br{v,e_k}\ge0}$.)
We define a map $M_k$ on walls that fixes any walls contained in $e_k^\perp$ and otherwise has
\begin{equation}
M_k(\d,f_\d)=\begin{cases}
M_k^-(\d,f_\d)&\text{if }\d\subseteq \set{v\in V^*:\br{v,e_k}\le0},\text{ or}\\
M_k^+(\d,f_\d)&\text{if }\d\subseteq \set{v\in V^*:\br{v,e_k}\ge0},\\
\end{cases}
\end{equation}
where 
\begin{align}  
\label{Mk+}
M_k^-\bigl(\d,f_\d(\zeta^{n_0})\bigr)&=\bigl(\d(J_k+[-B^{k\bullet}]_+),f_\d(\tilde\zeta^{(J_k+[(B^T)^{k\bullet}]_+)n_0})\bigr)\\
\label{Mk-}
M_k^+\bigl(\d,f_\d(\zeta^{n_0})\bigr)&= \bigl(\d(J_k+[B^{k\bullet}]_+),f_\d(\tilde\zeta^{(J_k+[(-B^T)^{k\bullet}]_+)n_0})\bigr).
\end{align}
Here $f_\d(a)$ means the result of replacing $\zeta^{n_0}$ by $a$ in the formal power series $f_\d(\zeta^{n_0})$, and $\tilde\zeta=(\tilde\zeta_i:i\in I_\uf)$ are the quantities defined in Table~\ref{init data} using $\mu_k(\tilde B)$.

We reuse the symbol $M_k$ for a map on scattering diagrams by defining 
\begin{equation}
M_k(\D)=\set{M_k(\d,f_\d):(\d,f_\d)\in\D}.
\end{equation}

\begin{thm}\label{mut thm}
$\Scat(\mu_k(\tB),\s)$ is equivalent to $M_k(\Scat(\tB,\s))$ for any extended exchange matrix~$\tB$ and any $k\in I_\uf$.
\end{thm}

Theorem~\ref{mut thm} will follow by simple (piecewise) linear algebra from \cite[Theorem~1.24]{GHKK}, which we quote below.
For each $k\in I_\uf$, we construct a scattering diagram $T_k(\Scat(\tB,\s))$ as follows:
The wall $(e_k^\perp,1+\zeta_k)$ is replaced by $(e_k^\perp,1+\zeta_k^{-1})$.
As in the definition of $M_k$, we can assume that $e_k^\perp$ intersects the relative interior of no other wall of $\Scat(\tB,\s)$.
If $\br{v,e_k}\le0$ for all points $v\in\d$, then $(\d,f_\d(\zeta^{n_0}))$ is fixed.
If $\br{v,e_k}\ge0$ for all points $v\in\d$, then $(\d,f_\d(\zeta^{n_0}))$ is replaced by $(\d',f_\d(\zeta^{n_0+\br{p^*(n_0),d_ke_k}e_k}))$, where $\d'$ is the image of $\d$ under the linear map $v\mapsto v+v_k\br{m,d_ke_k}$.
(To compare this with \cite[(1.23)]{GHKK}, it is useful to notice that $p^*(n_0+\br{p^*(n_0),d_ke_k}e_k)=p^*(n_0)+v_k\br{p^*(n_0),d_ke_k}$.)

Define $\mu_k(\s)=(e_i':i\in I)$ by
\begin{equation}
e'_i=\begin{cases}
e_i+[\epsilon_{ik}]_+e_k&\text{if }i\neq k,\text{ or}\\
-e_k&\text{if }i=k.
\end{cases}
\end{equation}
The following is \cite[Theorem~1.24]{GHKK}.

\begin{thm}\label{mut GHKK}
$\Scat(\mu_k(\tB),\mu_k(\s))$ is equivalent to $T_k(\Scat(\tB,\s))$ for any extended exchange matrix~$\tB$ and any $k\in I_\uf$.
\end{thm}

\begin{proof}[Proof of Theorem~\ref{mut thm}]
An easy computation shows that (as pointed out in \cite[Section~2]{GHK}), the basis $(d_i^{-1}(e'_i)^*,i\in I)$ for $M^\circ$ is $(f_i':i\in I)$, where 
\begin{equation}
f'_i=\begin{cases}
f_i&\text{if }i\neq k,\text{ or}\\
-f_k+\sum_{j\in I}[-\epsilon_{kj}]_+f_j&\text{if }i=k.
\end{cases}
\end{equation}
Theorem~\ref{mut GHKK} tells us how to write the cluster scattering diagram in terms of the basis $\mu_k(\s)=(e_i':i\in I)$ when $\set{\,\cdot\,,\,\cdot\,}$ is described in the basis $\mu_k(\s)$ by the matrix $\mu_k(\tB)$.
To get the scattering diagram $\Scat(\mu_k(\tB),\s)$, we need to rewrite this in terms of the basis $\s$ by applying the linear map that sends each $f_i'$ to $f_i$ and applying the corresponding map to monomials.
In the basis of the $f_i$, this map fixes $f_i$ for $i\neq k$ and sends $f_k$ to $-f_k+\sum_{j\in I}[-\epsilon_{kj}]_+f_j$.
The map on monomials sends each $\zeta^{e'_i}$ to $\zeta^{e_i}$ and thus sends $\zeta^{e_k}$ to $\zeta^{-e_k}$ and $\zeta^{e_i}$ to $\zeta^{e_i+[\epsilon_{ik}]_+e_k}$ for $i\neq k$. 

The wall $(e_k^\perp,1+\zeta_k)$ is sent by $T_k$ to $(e_k^\perp,1+\zeta_k^{-1})$ and the basis change restores it to $(e_k^\perp,1+\zeta_k)$.
In what remains, take $(\d,f_\d)$  to be a different wall in $\Scat(\tB,\s)$.

If $\br{v,e_k}\le0$ for all points in $\d$, then $T_k(\d,f_\d)=(\d,f_\d)$.
The basis change sends $(\d,f_\d)$ to $(\d',f_{\d'})$, where $\d'$ is the image of $\d$ under the map $v\mapsto v+\br{v,d_ke_k}{(-2f_k+\sum_{i\in I}[-\epsilon_{ki}]_+f_i)}$ and $f_{\d'}$ is obtained from $\d$ by replacing each monomial $\zeta^n$ by $\zeta^{n'}$, where $n'=n +(-2\br{e^*_k,n}+\sum_{i\in I}\br{e^*_i,n}[\epsilon_{ik}]_+)e_k$.

If $\br{v,e_k}\ge0$ for all points in $\d$, then the wall in $T_k(\d,f_\d)$ is the image of $\d$ under the map $v\mapsto v+v_k\br{m,d_ke_k}$.
Since $v_k=p^*(e_k)=\sum_{i\in I}\epsilon_{ki}f_i$, this map is ${v\mapsto v+\br{v,d_ke_k}\sum_{i\in I}\epsilon_{ki}f_i}$.
The change of basis sends this to a wall $\d'$ that is the image of $\d$ under the map $v\mapsto v+\br{v,d_ke_k}(-2f_k+\sum_{i\in I}[\epsilon_{ki}]_+f_i)$.
The function on $T_k(\d,f_\d)$ is obtained from $f_\d$ by replacing each $\zeta^{e_i}$ by $\zeta^{e_i+\br{p^*(e_i),d_ke_k}e_k}=\zeta^{e_i+\epsilon_{ik}e_k}$.
After the basis change, this sends $\zeta^{e_k}$ to $\zeta^{-e_k}$ and for $i\neq k$ sends $\zeta^{e_i}$ to $\zeta^{e_i+[-\epsilon_{ik}]_+e_k}$.
We have almost established that $\Scat(\mu_k(\tB),\s)$ is equivalent to $M_k(\Scat(\tB,\s))$, except that we are missing the ``tildes'' from $\tilde\zeta$ in \eqref{Mk+} and \eqref{Mk-}.
However, we can insert the tildes by Proposition~\ref{coeffs don't matter}.
\end{proof}

We can also think of $M_k$ as a piecewise linear map on $V^*$:
\begin{equation}
M_k(v)=\begin{cases}
M_k^-(v)=v\cdot(J_k+[-B^{k\bullet}]_+)&\text{if }\br{v,e_k}\le0,\text{ or}\\
M_k^+(v)=v\cdot(J_k+[B^{k\bullet}]_+)&\text{if }\br{v,e_k}\ge0.\\
\end{cases}
\end{equation}
We now explain how to describe the geometric action of $M_k$ in terms of matrix mutation.
As before, write $B$ for the exchange matrix ${[\epsilon_{ij}]_{i,j\in I_\uf}}$.

Given $B$ and a sequence $\kk$ of indices in $I_\uf$, the \newword{mutation map} $\eta_\kk^B:V^*\to V^*$ is defined as follows:
Given a vector $v=\sum_{i\in I_\uf}a_if_i\in V^*$, let $\mathring{B}$ be obtained from $B$ by adjoining $(a_i:i\in I_\uf)$ as an additional row below $B$.
Writing $(a'_i:i\in I_\uf)$ for the bottom row of $\mu_\kk(\mathring{B})$, we define $\eta_\kk^B(v)$ to be $\sum_{i\in I_\uf}a'_if_i\in V^*$.
This is a piecewise linear homeomorphism from $V^*$ to itself.
We write $\eta_k^B$ for $\eta_\kk^B$ when $\kk$ consists of a single index $k$.
In this case, for each $j\in I_\uf$:  
\begin{equation}\label{mutation map def}
a'_j=\begin{cases}
-a_k&\mbox{if }j=k;\\
a_j+a_kb_{kj}&\mbox{if $j\neq k$, $a_k\ge 0$ and $b_{kj}\ge 0$};\\
a_j-a_kb_{kj}&\mbox{if $j\neq k$, $a_k\le 0$ and $b_{kj}\le 0$};\\
a_j&\mbox{otherwise.}
\end{cases}
\end{equation}
The geometric action of $M_k$ on walls coincides with the mutation map $\eta^B_k$.
Since the hyperplane $e_k^\perp$ is a wall of $\Scat(\tB,\s)$, the action of $\eta^B_k$ is linear on all cones of $\ScatFan(B)$, and therefore for each sequence $\kk$, the action of $\eta^B_\kk$ is linear on all cones of $\ScatFan(B)$.
We record the following immediate corollary to Theorem~\ref{mut thm}.

\begin{cor}\label{mut scat fan}
For any exchange matrix~$B$ and any sequence $\kk$ of indices in $I_\uf$, the mutation map $\eta^B_\kk$ is a piecewise-linear isomorphism from $\ScatFan(B)$ to $\ScatFan(\mu_\kk(B))$.
\end{cor}

By definition, each hyperplane $e_i^\perp$ in $V^*$ is a wall of $\Scat(\tB,\s)$.
Since each wall of $\Scat(\tB,\s)$ is in the hyperplane perpendicular to a vector in $N^+$, no wall intersects the interior of the cone $\C^+=\set{p\in V^*:\br{p,e_i}\ge0\,\forall i\in I_\uf}$.
Since also each hyperplane $e_i^\perp$ is a wall, $\C^+$ is a cone in $\ScatFan(B)$.
Corollary~\ref{mut scat fan} says that for each $k\in I_\uf$, the map $\eta_k^B$ is a piecewise-linear isomorphism from $\ScatFan(B)$ to $\ScatFan(\mu_k(B))$.
Since $\C^+$ is also a cone of $\ScatFan(\mu_k(B))$ and since $\eta_k^B$ fixes the facet of $\C^+$ contained in $e_k^\perp$, we conclude that $\C^+$ shares this facet with another full-dimensional cone of $\ScatFan(B)$.

Two full-dimensional cones are \newword{adjacent} if they share a facet and \newword{transitively adjacent} if they are related in the transitive closure of the adjacency relation.
By induction, we see that every cone transitively adjacent to $\C^+$ in $\ScatFan(B)$ shares each of its facets with another cone of $\ScatFan(B)$.
Let $\ChamberFan(B)$ be the subfan of $\ScatFan(B)$ consisting of maximal cones that are transitively adjacent to~$\C^+$, together with all of their faces.
This fan is simplicial, and by \cite[Theorem~0.8]{GHKK}, it coincides with the fan of Fock-Goncharov cluster chambers and is isomorphic to the cluster complex associated to $B$.

As a further consequence of Corollary~\ref{mut scat fan}, we have the following.

\begin{cor}\label{mut clus fan}
For any exchange matrix~$B$ and any sequence $\kk$ of indices in $I_\uf$, the mutation map $\eta^B_\kk$ is a piecewise-linear isomorphism from $\ChamberFan(B)$ to $\ChamberFan(\mu_\kk(B))$.
\end{cor}

The following theorem is an easy consequence of results of \cite{GHKK}.
It is not stated explicitly in \cite{GHKK}, but is alluded to in a comment in \cite[Construction~4.1]{GHKK}.

\begin{theorem}\label{clus easy}
Suppose $F$ and $G$ are adjacent maximal cones of $\ChamberFan(B)$ and $n_0$ is the primitive normal to $F\cap G$ in $N^+$.
Then $f_p(\Scat(\tB,\s))=1+\zeta^{n_0}$ for every general point $p$ in $F\cap G$.
\end{theorem}
\begin{proof}
Let $\ell(F)$ be the smallest $\ell$ such that there exists a sequence $F_0,F_1,\ldots,F_\ell$ with $F_0=\C^+$ and $F_\ell=F$, such that $F_{i-1}$ and $F_i$ are adjacent for all $i=1,\ldots,\ell$.
We argue by induction on $\min(\ell(F),\ell(G))$.

In the base case, where this minimum is zero, the assertion is that, for each $i\in I_\uf$, that $f_p(\Scat(\tB,\s))=1+\zeta_i$ for every general point $p$ in $\C^+\cap e_i^\perp$.
This follows from \cite[Theorem~1.28]{GHKK}, as explained in \cite[Remark~1.29]{GHKK}.

If $\min(\ell(F),\ell(G))>0$, then without loss of generality, this minimum equals $\ell(F)$.
Consider a sequence of cones $\C^+=F_0,F_1,\ldots,F_{\ell(F)}=F$ with $F_{i-1}$ and $F_i$ adjacent for all $i=1,\ldots,\ell$.
Then $F_0\cap F_1$ is contained in $\e_k^\perp$ for some $k\in I_\uf$, and Corollary~\ref{mut scat fan} implies that $M_k(F_1)$ is the cone $\C^+$.
By induction, the assertion holds for $M_k(F)$ and $M_k(G)$, which are cones in $\ChamberFan(\mu_k(B))$, and then Theorem~\ref{mut thm} implies the assertion for $F$ and~$G$.
\end{proof}

\subsection{Cluster scattering fans refine mutation fans}\label{mut fan sec}
\begin{definition}[\emph{The mutation fan $\F_B$}]\label{mut fan def}
To define the mutation fan, we first define an equivalence relation $\equiv^B$ on $V^*$ by setting $v\equiv^Bv'$ if and only if, for every sequence $\kk$ of indices and every $i\in I_\uf$, the quantities $\br{\eta^B_\kk(v),e_i}$ and $\br{\eta^B_\kk(v'),e_i}$ have the same sign ($-1$, $0$,  or $1$).
The equivalence classes of $\equiv^B$ are called \newword{$B$-classes}.
The closures of $B$-classes are called \newword{$B$-cones}.
These are closed convex cones \cite[Proposition~5.4]{universal}, meaning that they are closed under nonnegative scaling and addition.
The set of $B$-cones and their faces constitutes a complete fan \cite[Theorem~5.13]{universal} called the \newword{mutation fan}.
\end{definition}

A collection $X$ of vectors in $V^*$ is \newword{sign-coherent} if for any $i\in I_\uf$, the set $\br{X,e_i}$ contains no two numbers with strictly opposite signs.
(That is, $\br{x,e_i}\br{y,e_i}\ge0$ for any $x,y\in X$ and $i\in I_\uf$.)
The following is \cite[Proposition~5.30]{universal}.
\begin{prop}\label{contained Bcone}
A set $C\subseteq V^*$ is contained in some $B$-cone if and only if the set $\eta_\kk^B(C)$ is sign-coherent for every sequence $\kk$ of indices in $I_\uf$.
\end{prop}

A complete fan $\F$ \newword{refines} a complete fan $\F'$ if every cone of $\F$ is contained in a cone of $\F'$, or equivalently if every cone of $\F'$ is a union of cones of $\F$.

\begin{theorem}\label{direct construct FB}
Consider a family of fans in $V^*$, with one fan $\E_{B'}$ for each $B'$ obtained as $\mu_\kk(B)$ for some sequence $\kk$ of indices.
Suppose that each cone of each $\E_{B'}$ is a sign-coherent set.
Suppose also that, for each $B'$, each sequence $\kk$, and each cone $C$ of $\E_{B'}$, the set $\eta_\kk^{B'}(C)$ is contained in a cone of $\E_{\mu_\kk(B')}$.
Then $\E_B$ refines the mutation fan~$\F_B$.
\end{theorem}
\begin{proof}
Let $C$ be a cone of $\E_B$.
By hypothesis, for any sequence $\kk$, the set $\eta_\kk^B(C)$ is contained in some cone of $\E_{\mu_\kk(B)}$ and thus $\eta_\kk^B(C)$ is sign-coherent.
Now Proposition~\ref{contained Bcone} says that~$C$ is contained in some $B$-cone.
\end{proof}

The following theorem is the main result of this section.

\begin{theorem}\label{scat ref mut}
The scattering fan $\ScatFan(B)$ refines the mutation fan $\F_B$ for any exchange matrix~$B$.
\end{theorem}

\begin{proof}
Since each hyperplane $e_i^\perp$ for $i\in I_\uf$ is a wall of the cluster scattering diagram, each cone of $\ScatFan(\mu_\kk(B))$ is a sign-coherent set.
Corollary~\ref{mut scat fan} says that for each $B'$, each sequence $\kk$, and each cone $C$ of $\ScatFan(B')$, the set $\eta_\kk^{B'}(C)$ \emph{is} a cone of $\ScatFan(\mu_\kk(B'))$.
By Proposition~\ref{direct construct FB}, $\ScatFan(B)$ refines $\F_B$.
\end{proof}

We advance the following conjecture on when the the cluster scattering fan and the mutation fan coincide.

\begin{conj}\label{scat eq mut}
Given an exchange matrix $B$, the scattering fan $\ScatFan(B)$ coincides with the mutation fan $\F_B$ if and only if either 
\begin{enumerate}[(i)]
\item
$B$ is $2\times2$ and of finite or affine type or 
\item
$B$ is $n\times n$ for $n>2$ and of finite \emph{mutation} type.
\end{enumerate}
\end{conj}

An exchange matrix $B$ is of \newword{finite mutation type} if only finitely many distinct matrices can be obtained from $B$ by sequences of mutations.
A comparison of results of \cite[Section~3]{CGMMRSW} and \cite[Section~9]{universal} confirms the conjecture when $B$ is $2\times 2$.
(See also \cite[Section~3]{scatcomb}.)  
When $B$ is of finite type, both $\ScatFan(B)$ and $\F_B$ are known to coincide with the $\g$-vector fan for $B^T$, so they coincide.
(See \cite[Lemma~2.10]{GHKK}, \cite[Corollary~5.9]{GHKK}, and \cite[Section~10]{universal}.
The transpose is not mentioned in \cite{GHKK}, but arises from a difference in conventions as explained in \cite[Section~1]{scatcomb}.)
The conjecture also seems plausible in light of a characterization of mutation-finiteness given in \cite[Theorem~2.8]{FeShT}:
For $n>2$, an $n\times n$ exchange matrix is mutation finite if and only if for all $i\neq j$, the matrix $\begin{bsmallmatrix}b_{ii}&b_{ij}\\b_{ji}&b_{jj}\end{bsmallmatrix}=\begin{bsmallmatrix}0&b_{ij}\\b_{ji}&0\end{bsmallmatrix}$ is of finite or affine type (i.e.\ $b_{ij}b_{ji}\ge-4$).

\section{Cluster monomials}\label{broken sec}
We now discuss two ways in which cluster monomials can be described in terms of the piecewise-linear geometry of the cluster scattering diagram: in terms of broken lines and in terms of path-ordered products.
Readers unfamiliar with cluster algebras may wish to take Theorem~\ref{clus mon thm} as a definition of cluster monomials.

Suppose $\D$ is a scattering diagram. 
Given $m_0\in M^\circ\setminus\set{0}$ and ${Q\in V^*\setminus\Supp(\D)}$, a \newword{broken line} for $m_0$ with endpoint $Q$ is a piecewise linear path $\gamma:(-\infty,0]\to V^*$ with a finite number of domains of linearity, with each domain $L$ of linearity marked with a monomial $c_Lz^{m_L}$ for $c_L\in\k$ and $m_L\in M^\circ$, satisfying the following properties:
\begin{enumerate}[(i)]
\item $\gamma(0)=Q$.
\item \label{brok generic}
$\gamma$ does not intersect the relative boundary of any wall and does not intersect any two non-parallel walls at the same point. 
\item For each domain $L$ of linearity, if $m_L=\sum_{i\in I}a_if_i$, then $\gamma'$ is constantly equal to $-\sum_{i\in I_\uf}a_if_i$ in $L$.
(Notice the difference in indexing sets for the sums.)
\item If $L$ is the unbounded domain of linearity of $\gamma$, then $c_L=1$ and $m_L=m_0$.
\item Suppose $t\in(-\infty,0)$ is a point where $\gamma$ is not linear and $L$ is the domain of linearity consisting of values just smaller that $t$ while $L'$ is the domain of linearity consisting of values just larger than $t$.
By \eqref{brok generic}, there exists $n$ primitive in $N^\circ\cap N_\uf$ such that all walls containing $\gamma(t)$ are in $n^\perp$ and such that $\br{m_L,n}>0$.
Let $f$ be the product of the functions $f_\d$ for all walls $(\d,f_\d)$ with $\gamma(t)\in\d$.
Then $c_{L'}z^{m_{L'}}$ equals $c_Lz^{m_L}$ times a term in the formal power series $f^{\br{m_L,n}}$.
\end{enumerate}
Write $c_\gamma z^{m_\gamma}$ for the monomial on the domain of linearity containing $0$ and define the \newword{theta function} $\thet_{Q,m_0}$ to be the sum of the $c_\gamma z^{m_\gamma}$, over all broken lines $\gamma$ for $m_0$ with endpoint $Q$.
This may be an infinite sum, but makes sense as an element of $z^{m_0}\k[[\zeta]]$, as verified in \cite[Proposition~3.4]{GHKK}.

\begin{remark}\label{useless dimensions broken lines}
The definition of broken lines here differs from that in \cite[Definition~3.1]{GHKK} because, as explained in Remark~\ref{useless dimensions}, we define our scattering diagrams in a lower-dimensional vector space.
It is easy to verify that there is a bijection between the broken lines defined here and broken lines in the larger space in the sense of \cite[Definition~3.1]{GHKK} and that both definitions define the same functions~$\thet_{Q,m_0}$.
\end{remark}

Say a point $Q\in V^*$ is \newword{generic} if it is not contained in any hyperplane $n^\perp$ for $n\in N_\uf$.
Write $M^\circ_\uf$ for the saturated sublattice of $M^\circ$ given by the $\integers$-linear span of $\set{f_i:i\in I_\uf}$.
For $m_0\in M_\uf^\circ$, we will write $m_0\in\ChamberFan(B)$ to mean that $m_0$ is contained in some cone in $\ChamberFan(B)$.

The \newword{frozen variables} are the indeterminates $\set{z_i:i\in I_\fr}$.
Given a Laurent polynomial $\thet$ in the $z_i$ and given $m\in M^\circ$, say that $z^m$ \newword{clears frozen denominators} of $\thet$ if all terms of $z^m\thet$ have nonnegative exponents on all frozen variables.
There is a unique minimum (componentwise in the basis of the $f_i$) $m$ such that $z^m$ clears frozen denominators of $\thet$, and $m$ is of the form $\sum_{i\in I_\fr}c_if_i$ with the $c_i$ nonnegative.
We write $\Clear_\fr(\thet)$ for $z^m\thet$ when $m$ is this unique minimum.

The following is an adaptation of \cite[Theorem~4.9]{GHKK}.

\begin{theorem}\label{clus mon thm}
Let $Q$ be a generic point in $\C^+$.
Then \[\bigl\lbrace\Clear_\fr(\thet_{Q,m_0}):m_0\in M_\uf^\circ\cap\ChamberFan(B)\bigr\rbrace\] is the set of cluster monomials in the cluster algebra defined by $\tB$.
\end{theorem}

\begin{remark}\label{what are clus mons}
The differences between our Theorem~\ref{clus mon thm} and \cite[Theorem~4.9]{GHKK} are superficial, but they convey important differences in point of view.
The definition of cluster monomials in \cite[Definition~4.8]{GHKK} allows the frozen variables, as is not uncommon.
However, instead of taking ordinary monomials in the cluster variables, frozen and unfrozen, in some cluster, \cite[Definition~4.8]{GHKK} allows cluster variables to be ordinary monomials in the unfrozen variables in the cluster but \emph{Laurent} in the frozen variables.
We are taking a more restrictive definition of cluster monomials: ordinary monomials \emph{in the unfrozen variables only} in some cluster.
Merely taking theta functions for $m_0\in M_\uf$ does not work because negative powers of the frozen variables can occur in $\thet_{Q,m_0}$.
Applying the operator $\Clear_\fr$ is the correct fix, in light of \cite[Proposition~5.2]{ca4}.
Indeed, without this correction, $\thet_{m_0}$ is ``missing'' the denominator of the formula given in \cite[Corollary~6.3]{ca4}.

In any case, it is easy to insert or delete powers of the frozen variables in theta functions.
Given a vector $m\in M^\circ$, write $m=\sum_{i\in I}a_if_i$ and define $m_\uf=\sum_{i\in I_\uf}a_if_i$ and $m_\fr=m-m_\uf$.
For all $m\in M^\circ$, we have $\thet_{Q,m}=z^{m_\fr}\thet_{Q,m_\uf}$.
\end{remark}

\begin{remark}\label{froz clus mon}
In \cite{GHKK}, the algebra $\can(\tB,\s)$ of $\k$-linear combinations of functions $\thet_{Q,m_0}$ (for fixed generic $Q\in\C^+$) is defined.
Under certain circumstances, explored at length in \cite{GHKK}, the algebra $\can(\tB,\s)$ coincides with the upper cluster algebra \cite{ca3}, or even with the cluster algebra \cite{ca4}.
Writing $\k[z^\pm_\fr]$ for the Laurent polynomial ring $\k[z_i^\pm:i\in I_\fr]$, the algebra $\can(\tB,\s)$ is the set of $\k[z^\pm_\fr]$-linear combinations of theta functions $\thet_{Q,m_0}$ with $m_0\in M^\circ_\uf$.
There are advantages and disadvantages in different settings to inverting the frozen variables $\set{z_i:i\in I_\fr}$.
In settings where we wish not to invert the frozen variables, we can instead write $\k[z_\fr]$ for the polynomial ring $\k[z_i:i\in I_\fr]$ and consider the algebra of $\k[z_\fr]$-linear combinations of functions $\Clear_\fr(\thet_{Q,m_0})$ such that $m_0\in M^\circ_\uf$.
\end{remark}

\begin{remark}\label{suppress}
Theorem~\ref{change of basepoint}, below, makes it clear that $\thet_{Q,m_0}$ is independent of the choice of generic $Q\in\C^+$.
Thus, eventually it will make sense to suppress $Q$ from the notation, writing $\thet_{m_0}$ for $\thet_{Q,m_0}$.
\end{remark}

One can also describe cluster monomials in terms of path-ordered products.

\begin{theorem}\label{clus mon pop}
Fix $Q$ generic in $\C^+$.
If $m_0\in M_\uf^\circ\cap\ChamberFan(B)$, then $\thet_{Q,m_0}$ is $\p_{Q',Q,\Scat(\tB,\s)}(z^{m_0})$, where $Q'$ is any point in the interior of a maximal cone $F$ of $\ChamberFan(B)$ with $m_0\in F$.
\end{theorem}

Theorem~\ref{clus mon pop} relies on two results from \cite[Section~3]{GHKK} that we now quote.
These are \cite[Theorem~3.5]{GHKK} and \cite[Corollary~3.9]{GHKK}.

\begin{theorem}\label{change of basepoint}
Suppose $\D$ is a consistent scattering diagram, $m_0\in M^\circ$, and $Q$ and $Q'$ are generic points in $V^*$.
Then $\thet_{Q,m_0}=\p_{Q',Q,\D}(\thet_{Q',m_0})$.
\end{theorem}

\begin{theorem}\label{easy theta}
Suppose $F$ is a maximal cone in $\ChamberFan(B)$ and suppose $Q$ is in the interior of $F$ and $m\in M^\circ_\uf$ is in $F$.
Then $\thet_{Q,m}=z^m$.
\end{theorem}

\begin{proof}[Proof of Theorem~\ref{clus mon pop}]
Since $\p_{Q',Q,\Scat(\tB,\s)}$ is independent of the choice of $Q'$ within the interior of $F$, we may as well take $Q'$ generic.
Theorem~\ref{change of basepoint} says that $\thet_{Q,m_0}=\p_{Q',Q,\Scat(\tB,\s)}(\thet_{Q',m_0})$, which equals $\p_{Q',Q,\Scat(\tB,\s)}(z^{m_0})$ by Theorem~\ref{easy theta}.
\end{proof}

\section*{Acknowledgments}
Thanks to Man Wai ``Mandy'' Cheung, Greg Muller, and Salvatore Stella for helpful conversations.
Thanks also to Mark Gross and Paul Hacking for helpful answers to questions.
Finally, thanks to the MSRI for running the Hot Topics Workshop on Cluster algebras and wall-crossing, March--April, 2016.

\end{document}